\documentclass[11pt]{article}
\usepackage[margin=1in]{geometry}
\usepackage{amsmath,amssymb,amsthm,mathtools}
\usepackage{microtype}

\usepackage{float}
\usepackage{tikz}
\usepackage{pgfplots}
\pgfplotsset{compat=1.16}
\usetikzlibrary{calc,arrows.meta,decorations.pathreplacing}
\usepackage[colorlinks=true,linkcolor=blue,citecolor=blue,urlcolor=blue,
  pdftitle={Explicit finite size corrections for Toeplitz determinants and
            rectangular domino tilings},
  pdfauthor={Miguel Tierz},
  pdfsubject={Toeplitz and Hankel determinants, dimer models},
  pdfkeywords={Toeplitz determinant, Hankel matrix, banded Toeplitz,
    finite section, Borodin--Okounkov--Geronimo--Case identity, strong Szego theorem,
    domino tiling, dimer model, Kasteleyn, Schur polynomial}]{hyperref}
\usepackage[nameinlink,noabbrev]{cleveref}

\definecolor{bandcol}{RGB}{219,229,240}
\definecolor{cornercol}{RGB}{186,62,50}
\definecolor{cornerfill}{RGB}{246,211,205}
\definecolor{imgcol}{RGB}{115,115,124}
\definecolor{accent}{RGB}{28,80,133}

\newtheorem{theorem}{Theorem}[section]
\newtheorem{lemma}[theorem]{Lemma}
\newtheorem{proposition}[theorem]{Proposition}
\newtheorem{corollary}[theorem]{Corollary}
\theoremstyle{remark}
\newtheorem{remark}[theorem]{Remark}

\newcommand{\ii}{\mathrm{i}}
\newcommand{\bbC}{\mathbb{C}}
\newcommand{\bbT}{\mathbb{T}}
\newcommand{\cO}{\mathcal{O}}

\title{Explicit finite size corrections for Toeplitz determinants and rectangular domino tilings}
\author{Miguel Tierz\\[3pt]
{\small Shanghai Institute for Mathematics and Interdisciplinary Sciences (SIMIS),}\\
{\small Block A, International Innovation Plaza, No.~657 Songhu Road,}\\
{\small Yangpu District, Shanghai 200438, China}\\[3pt]
{\small\texttt{tierz@simis.cn}}}
\date{}

\begin{document}
\maketitle
\begin{abstract}
Finite Toeplitz truncation does not, in general, preserve multiplication.
For products of \(m\) \(N\times N\) tridiagonal Toeplitz matrices with symbols
\(z+z^{-1}+a_h\), we make the two boundary corrections in Widom's product
formula for finite sections explicit.  The method of images gives two Hankel
corrections when \(m\le N+2\), a sharp uniform degree bound.  For
\(N\ge m-1\), the ratio of the determinant of the matrix product to the
Toeplitz determinant with the product symbol reduces, whenever defined, to a
determinant of order \(2(m-1)\), which reversal symmetry splits into two
determinants of order \(m-1\).  A rectangular Schur polynomial gives an
exact formula for the Toeplitz determinant with the product symbol, including
when roots coincide.  For symbols nonzero on the unit circle, classical identities for rational
symbols give an exact \(m\times m\) determinant formula for the ratio, with an
explicit confluent form for repeated roots.  With \(\rho_h\) the zero of
\(z+z^{-1}+a_h\) in \(|z|<1\), the limit is
\(\prod_{i<j}(1-\rho_i\rho_j)^2\) as \(N\to\infty\) at fixed \(m,a_h\).
For rectangular domino tilings, these formulas compare the classical
Kasteleyn--Temperley--Fisher product with the banded Toeplitz determinant
formed from the product symbol.  We obtain the first exponentially small
correction at fixed width and relate this comparison to the Toeplitz
determinant furnished by the dual Cauchy identity.
\end{abstract}

\section{Introduction}
Let \(N\ge1\), and let
\(
G(z)=\sum_{k\in\mathbb Z}g_kz^k
\)
be a Laurent polynomial, meaning that only finitely many coefficients \(g_k\)
are nonzero.  Its \(N\times N\) Toeplitz matrix is
\begin{equation}\label{eq:TN_definition_intro}
T_N(G):=[g_{i-j}]_{i,j=1}^N,
\qquad
D_N(G):=\det T_N(G).
\end{equation}
Thus the entries of \(T_N(G)\) are constant along diagonals.  Equivalently,
it is the principal \(N\times N\) block of the doubly infinite matrix
\([g_{i-j}]_{i,j\in\mathbb Z}\).  The Laurent polynomial \(G\) is called the
\emph{symbol}, and \(G\mapsto T_N(G)\) is called finite Toeplitz truncation
(or passage to a finite Toeplitz section).

This truncation is linear but is not, in general, multiplicative:
\[
T_N(f_1)\cdots T_N(f_m)\ne T_N(f_1\cdots f_m).
\]
For two symbols \(f(z)=\sum_k f_kz^k\) and
\(g(z)=\sum_k g_kz^k\), the reason is already visible entrywise:
\[
\bigl(T_N(f)T_N(g)\bigr)_{ij}
=\sum_{\ell=1}^{N}f_{i-\ell}g_{\ell-j},
\qquad
\bigl(T_N(fg)\bigr)_{ij}
=\sum_{\ell\in\mathbb Z}f_{i-\ell}g_{\ell-j}.
\]
The first sum omits the intermediate indices outside
\(\{1,\dots,N\}\).  Widom's product formula for finite sections expresses the
discrepancy as the sum of two Hankel products
\cite[Eq.~(1.4)]{Widom76}.  See also
\cite[Prop.~3.10]{BG05}.  Thus one correction comes from each finite
boundary.  For the even tridiagonal symbols considered here, the two Hankel
products in the formula for two factors are rank one.  We make the corrections
explicit for an arbitrary number of factors and then determine their effect
on the determinant.

Rectangular domino tilings provide the application: their partition function
has both the Kasteleyn--Temperley--Fisher product formula
\cite{Kas61,TF61,Fisher61} and a Toeplitz determinant formula derived from the
dual Cauchy identity \cite{Macdonald95,Gessel90}.

Fix \(m\ge1\) and parameters \(a_1,\dots,a_m\in\bbC\).  Define
\[
f_h(z):=z+z^{-1}+a_h,\qquad
L_N:=T_N(z+z^{-1}),\qquad
P(x):=\prod_{h=1}^m(x+a_h),\qquad
F_m(z):=P(z+z^{-1}).
\]
The matrix \(L_N\) has entries \(1\) when \(|i-j|=1\) and \(0\) otherwise.
It is the adjacency matrix of the path with vertices \(1,\dots,N\).
Consequently,
\(
T_N(f_h)=L_N+a_hI_N
\), where \(I_N\) is the \(N\times N\) identity matrix.  These factors
commute.  With the usual polynomial evaluation
\(
P(L_N)=\sum_{\ell}p_\ell L_N^\ell
\)
when \(P(x)=\sum_\ell p_\ell x^\ell\), we therefore have
\[
P(L_N)=\prod_{h=1}^mT_N(f_h),
\]
whereas \(T_N(F_m)\) is obtained by multiplying the Laurent symbols first
and truncating only afterwards.

To state their difference precisely, for
\(G(z)=\sum_kg_kz^k\) define
\begin{equation}\label{eq:HN_JN_definition_intro}
H_N(G):=[g_{i+j}]_{i,j=1}^N,
\qquad
(J_N)_{ij}:=
\begin{cases}
1,&i+j=N+1,\\
0,&i+j\ne N+1.
\end{cases}
\end{equation}
The matrix \(H_N(G)\) is Hankel: its entries are constant along
antidiagonals.  With indices starting at \(1\), positive Laurent coefficients
enter \(H_N(G)\).  For a banded symbol, its nonzero entries begin at the
upper left corner.  The matrix \(J_N\) reverses the order of the indices, so
\[
J_NH_N(G)J_N
=[g_{2(N+1)-i-j}]_{i,j=1}^N
\]
is the corresponding lower right term.  In these conventions the identity
for banded symbols used throughout the paper is
\begin{equation}\label{eq:ToeplitzHankel_intro}
P(L_N)=T_N(F_m)-H_N(F_m)-J_NH_N(F_m)J_N,
\qquad m\le N+2.
\end{equation}
The condition \(m\le N+2\) is sharp when the identity is required uniformly
for all polynomials of degree \(m\).  The determinant reduction below uses
the slightly stronger condition \(N\ge m-1\).  Indeed, if
\(
F_m(z)=\sum_{k=-m}^mc_kz^k
\), then \(c_k=0\) for \(|k|>m\), and \(c_{i+j}\) can be nonzero only when
\(i,j\le m-1\).  Thus the two corrections in
\eqref{eq:ToeplitzHankel_intro} are supported in the first and last
\(m-1\) coordinates and have combined rank at most \(2(m-1)\), independently
of \(N\).  Although the enclosing coordinate blocks overlap when
\(N<2(m-1)\), their disjointness is not needed.  \Cref{fig:corners} shows the resulting pattern.

The standard sine vectors
\[
v^{(r)}_j=\sin\frac{jr\pi}{N+1},
\qquad 1\le j,r\le N,
\]
diagonalize \(L_N\), with eigenvalues
\(2\cos(r\pi/(N+1))\).  They therefore diagonalize every \(P(L_N)\).
Explicit formulas for the entries of positive integer powers of tridiagonal
Toeplitz matrices are given in \cite{Gutierrez08}, and a binomial form for
sufficiently large matrix order appears in \cite{Gutierrez13}.

The matrices diagonalized by this discrete sine transform of type I form the
matrix algebra \(\tau_N\) \cite{BiniCapovani83}.  A representation of a
matrix in \(\tau_N\) as the difference of a Toeplitz matrix and a Hankel matrix
is given in
\cite[Sec.~2, Eq.~(3)]{HendrickxVanBarel02}.  The same work represents
a symmetric banded Toeplitz matrix using two corner corrections and splits
the associated boundary system by reversal symmetry
\cite[Eqs.~(7) and (9)]{HendrickxVanBarel02}.  A more recent description of the
algebra generated by a finite tridiagonal matrix uses a Toeplitz matrix and
two opposite Hankel corrections
\cite[Sec.~5.1, pp.~698--699]{BiniMeini25}.  When the two endpoint
parameters in that construction vanish, their formulas specialize, for
\(\deg P\le N-1\), to the representation in
\eqref{eq:ToeplitzHankel_intro}: one Toeplitz matrix minus two Hankel matrices.

For arbitrary bandwidth and fixed polynomial degree, a related localization
statement is recorded in
\cite[Sec.~III-D, pp.~1170--1171]{CrespoGutierrez07}: once \(N\) is
sufficiently large, a polynomial of a banded Toeplitz matrix is Toeplitz
outside the upper left and lower right corner blocks, whose sizes and entries
are independent of \(N\).  Entrywise convergence for continuous functions of
banded Hermitian block Toeplitz matrices and their perturbations, with
convergence estimates for analytic functions, is studied in
\cite{GutierrezCrespoBottcher07}.

The product formula for finite sections accounts for the two boundary terms.  An
extension to certain distributional symbols is given in
\cite[Prop.~4.4]{BasorEhrhardt01}.  For the semi-infinite path
operator, the corresponding identity with one boundary, expressed as
Toeplitz minus Hankel, is contained in the framework of compatible pairs
\cite[Prop.~3.1(III)]{BasorEhrhardt09}.  An exact determinant
identity for finite sections of these compatible pairs is given in
\cite[Prop.~4.1]{BasorEhrhardt09}.

Determinant asymptotics when a function is applied to a semi-infinite Toeplitz
operator before taking a finite section are treated separately in
\cite[Prop.~4.1 and Thm.~4.2]
{BasorEhrhardtVirtanen25}.  The order of the operations matters: the present
matrix is \(P(T_N(z+z^{-1}))\), so polynomial evaluation takes place after
finite truncation and both boundaries contribute.  An asymptotic formula for
\(\operatorname{tr} f(T_N(a))\), with \(f\) analytic, is given in
\cite[Thm.~6.2]{Widom76}.  A refinement of the remainder estimate for weighted
Wiener algebras appears in \cite[Thm.~1.5]{Karlovich07}.  These
asymptotic results do not give the exact matrix identity with two boundary
terms or the reduction for finite \(N\)
to determinants of dimension independent of \(N\).

Functions of finite difference matrices also exhibit a structure involving
Toeplitz and Hankel matrices together with reflection
\cite[Eqs.~(3.6) and (3.11)--(3.12)]{StrangMacNamara14}.  In work on spectral distribution, the
discrepancy between a product of two
finite Toeplitz matrices and the Toeplitz matrix with the product symbol is
known to have uniformly bounded rank when one factor is a Laurent polynomial
\cite[Rem.~3.3]{SerraCapizzanoSesanaStrouse10}.  We give a walk proof
to fix the indexing, cover overlapping boundary
blocks, and identify the uniform degree range.  The iteration in
Remark~\ref{rem:degree_sharp} proves that this range is sharp.  A general
result on boundary localization for Toeplitz quasicommutators shows, in a
multidimensional block setting, that elements of the quasicommutator ideal
have asymptotically no mass in the interior and determines which boundary
strata may carry mass
\cite{BottcherGutierrezCrespo07}.  This is an asymptotic localization result
rather than an exact finite identity in one dimension.  Determinants of sums
of Toeplitz and Hankel matrices form a much broader asymptotic class
\cite{DIK11,GharakhlooIts20}.

Exact formulas for determinants of certain
even matrices formed as the sum of a Toeplitz matrix and a Hankel matrix are
given in \cite{Elouafi15}.
Factorizations of full Toeplitz determinants with even symbols into
determinants of sums of Toeplitz and Hankel matrices, together with Fredholm
expansions for finite rank kernels associated with rational symbols, are developed in
\cite[Secs.~5--6]{BasorConrey24}.
More recently, exact formulas were obtained for finite matrices that are the
sum of a Toeplitz matrix and a Hankel matrix, both generated by a rational symbol
\cite{BasorMorrison26}.  Those formulas
concern a single Hankel summand, whereas \eqref{eq:ToeplitzHankel_intro}
contains corrections at both finite boundaries.

\begin{figure}[!htbp]
\centering
\begin{tikzpicture}[x=3.6mm,y=-3.6mm,line join=round]
\begin{scope}
  \def\N{14}\def\M{4}\def\K{3}\def\MP{5}
  \fill[bandcol]
    (0,0) -- (\M,0) -- (\N,{\N-\M}) -- (\N,\N) -- ({\N-\M},\N) -- (0,\M) -- cycle;
  \foreach \g in {0,...,14}{
    \draw[gray!30,line width=.15pt] (\g,0) -- (\g,\N);
    \draw[gray!30,line width=.15pt] (0,\g) -- (\N,\g);}
  \draw[black,line width=.7pt] (0,0) rectangle (\N,\N);
  \foreach \i in {1,...,\K}{\foreach \j in {1,...,\K}{
     \pgfmathparse{int(\i+\j)}\ifnum\pgfmathresult<\MP
       \fill[cornerfill] (\j-1,\i-1) rectangle (\j,\i);
       \draw[cornercol,line width=.2pt] (\j-1,\i-1) rectangle (\j,\i);
     \fi}}
  \foreach \i in {1,...,\K}{\foreach \j in {1,...,\K}{
     \pgfmathparse{int(\i+\j)}\ifnum\pgfmathresult<\MP
       \fill[cornerfill] (\N-\j,\N-\i) rectangle (\N-\j+1,\N-\i+1);
       \draw[cornercol,line width=.2pt] (\N-\j,\N-\i) rectangle (\N-\j+1,\N-\i+1);
     \fi}}
  \draw[cornercol,line width=.7pt,densely dashed] (0,0) rectangle (\K,\K);
  \draw[cornercol,line width=.7pt,densely dashed] (\N-\K,\N-\K) rectangle (\N,\N);
  \draw[decorate,decoration={brace,amplitude=3pt,mirror},cornercol,line width=.5pt]
        (-0.5,0) -- (-0.5,\K) node[midway,left=3pt,scale=.75,text=cornercol]{$k$};
  \draw[{Stealth[length=1.6mm]}-{Stealth[length=1.6mm]},accent!75!black,line width=.5pt]
        (3,7) -- (7,3);
  \node[scale=.72,fill=white,inner sep=1pt,text=accent!75!black] at (5,5) {$2m{+}1$};
  \node[scale=.8,rotate=-45,fill=white,inner sep=1pt,text=accent!75!black]
        at (9.6,9.6) {$T_N(F_m)$};
  \node[scale=.78,anchor=west,text=cornercol,fill=white,inner sep=1pt]
        at (4.6,1.35) {$-H_N(F_m)$};
  \node[scale=.78,anchor=east,text=cornercol,fill=white,inner sep=1pt]
        at (9.4,12.65) {$-J_NH_N(F_m)J_N$};
  \node[scale=.9] at (\N/2,-2.4) {$N\ge 2(m-1)$};
\end{scope}
\begin{scope}[xshift=72mm]
  \def\N{5}\def\M{4}\def\K{3}\def\MP{5}
  \fill[bandcol] (0,0) rectangle (\N,\N);
  \foreach \g in {0,...,5}{
    \draw[gray!30,line width=.15pt] (\g,0) -- (\g,\N);
    \draw[gray!30,line width=.15pt] (0,\g) -- (\N,\g);}
  \draw[black,line width=.7pt] (0,0) rectangle (\N,\N);
  \foreach \i in {1,...,\K}{\foreach \j in {1,...,\K}{
     \pgfmathparse{int(\i+\j)}\ifnum\pgfmathresult<\MP
       \fill[cornerfill] (\j-1,\i-1) rectangle (\j,\i);
       \draw[cornercol,line width=.2pt] (\j-1,\i-1) rectangle (\j,\i);
     \fi}}
  \foreach \i in {1,...,\K}{\foreach \j in {1,...,\K}{
     \pgfmathparse{int(\i+\j)}\ifnum\pgfmathresult<\MP
       \fill[cornerfill,opacity=.6] (\N-\j,\N-\i) rectangle (\N-\j+1,\N-\i+1);
       \draw[cornercol,line width=.2pt] (\N-\j,\N-\i) rectangle (\N-\j+1,\N-\i+1);
     \fi}}
  \draw[cornercol,line width=.7pt,densely dashed] (0,0) rectangle (\K,\K);
  \draw[cornercol,line width=.7pt,densely dashed] (\N-\K,\N-\K) rectangle (\N,\N);
  \node[scale=.9] at (\N/2,-2.4) {$N<2(m-1)$};
  \node[anchor=north,align=center,text=gray!50!black,font=\scriptsize]
       at (\N/2,6.0) {the outlined blocks overlap\\the shaded triangles do not};
\end{scope}
\end{tikzpicture}
\caption{The three terms of \eqref{eq:ToeplitzHankel_intro}, drawn for $m=4$,
so that $k:=m-1=3$.  The shaded band carries the at most $2m+1$ potentially
nonzero diagonals of $T_N(F_m)$, while the two outlined corners carry $-H_N(F_m)$
and $-J_NH_N(F_m)J_N$.  In coordinates measured from the relevant corner,
only entries with $i+j\le m$ can be nonzero, since $c_k=0$ for $|k|>m$,
and they are constant along antidiagonals.  When $N\ge m-1$, as in both panels, each corner fits in a
$k\times k$ block whose dimension does not grow with $N$.  This allows the
determinant ratio \eqref{eq:R_N_intro} to be represented by a
$2k\times2k$ determinant.  Left:
$N=14$.  Right: $N=5$, where $N<2(m-1)$ and the outlined blocks overlap,
while their nonzero triangular parts remain disjoint.}
\label{fig:corners}
\end{figure}
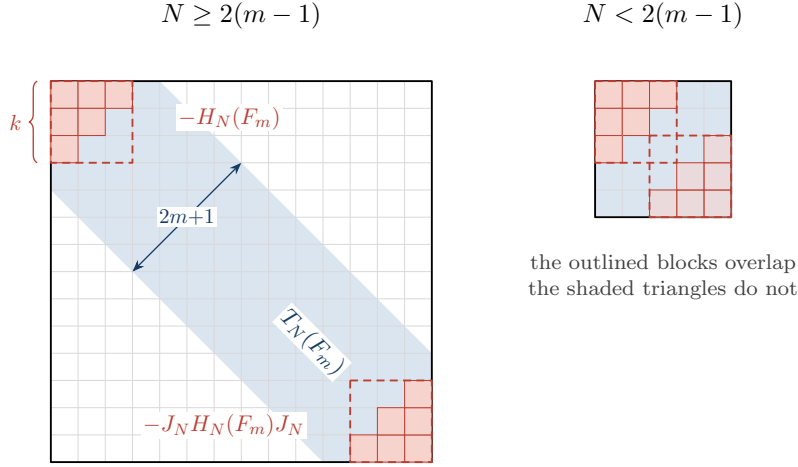

The identity reduces the determinant ratio to a determinant of fixed dimension.
Whenever
\(D_N(F_m)\ne0\), define
\begin{equation}\label{eq:R_N_intro}
R_N:=\frac{\det P(L_N)}{D_N(F_m)}.
\end{equation}
For \(N\ge m-1\), \eqref{eq:ToeplitzHankel_intro} expresses \(R_N\) as a
\(2(m-1)\times2(m-1)\) determinant, even when the two boundary blocks overlap.
Moreover, \(F_m(z)=F_m(z^{-1})\), so \(T_N(F_m)\) is unchanged when its row
and column indices are simultaneously reversed:
\(
J_NT_N(F_m)J_N=T_N(F_m)
\).  The reduced determinant therefore factors into a product of two
\((m-1)\times(m-1)\) determinants, obtained by restricting to the \(+1\) and
\(-1\) eigenspaces of the swap on
\(\bbC^{m-1}\oplus\bbC^{m-1}\).
Perturbations confined to
corners of fixed size, including exact determinant reductions and their limits
as \(N\to\infty\) under stability assumptions, are treated in
\cite[Eqs.~(8)--(10) and Thm.~3.1]{BFGM15}.

The inversion symmetry of the Laurent symbol also leads to a closed
expression in terms of reciprocal roots.
Let
\(
\bbT:=\{z\in\bbC:|z|=1\}
\)
denote the unit circle, and assume that \(z^2+a_hz+1\) has no zero on
\(\bbT\).  Its two roots have product \(1\), so exactly one lies in the open
unit disk.  Denote it by \(\rho_h\).  Equivalently,
\(
a_h=-(\rho_h+\rho_h^{-1})
\).
The change of variables \(x=z+z^{-1}\), sometimes called the
Dickson--Zhukovsky map, has been used for symmetric polynomials in variables
paired with their reciprocals and for minors of banded Toeplitz matrices with
symmetric Laurent symbols
\(G(z)=G(z^{-1})\) \cite{AlexanderssonEtAl21}.  For this family, applying
the strong Szeg\H{o} theorem \cite{Szego52} separately to the numerator and
denominator in \(R_N\) gives, for fixed \(m\),
\begin{equation}\label{eq:root_limit_intro}
R_N\longrightarrow \prod_{1\le i<j\le m}(1-\rho_i\rho_j)^2
\qquad (N\to\infty).
\end{equation}
For symmetric banded Toeplitz matrices, Chebyshev determinant formulas of
fixed order in terms of the reciprocal root pairs of the Laurent symbol are
available in \cite[Thm.~1]{Elouafi14}.  A specialization of the Toeplitz minor
identity in \cite[Eq.~(19)]{GGT20} expresses \(D_N(F_m)\) exactly in terms of
a Schur polynomial indexed by a rectangular partition, without requiring the
roots to be distinct.  For pairwise distinct roots we also use
an exact formula involving an \(m\times m\) determinant built from
\([(1-\rho_i\rho_j)^{-1}]_{i,j=1}^m\).  For this Laurent symbol with inversion
symmetry, it is a specialization of classical identities for rational symbols
\cite{Day75,BasorForrester94} and of the
Borodin--Okounkov--Geronimo--Case identity
\cite{GC79,BO00,BW00,Bot01,Bottcher02}.  Related Fredholm expansions for finite rank
kernels associated with rational symbols, including determinants of Cauchy
matrices of this type, appear in
\cite[Sec.~6]{BasorConrey24}.  The identity involving the Cauchy matrix is
first stated in a form that does not divide by \(D_N(F_m)\), so that it remains
valid when this Toeplitz determinant vanishes.  For repeated roots, the
corresponding matrix representation is obtained by taking a confluent limit.
Closely related confluent Vandermonde formulas for Schur and complete
homogeneous polynomials with repeated variables are developed in
\cite{GonzalezSerranoMaximenko25,GonzalezSerranoMaximenkoCH25}.

For the application, let \(Z_{m,n}\) be the number of tilings of a
\(2m\times2n\) board of unit squares by \(2\times1\) dominoes.  The
Kasteleyn--Temperley--Fisher product and the Toeplitz determinant derived from
the dual Cauchy identity are classical evaluations of \(Z_{m,n}\).  The
former is, up to the sign \((-1)^{mn}\), the determinant
\(\det P(L_{2n})\), obtained by diagonalizing \(L_{2n}\) with the discrete
sine transform.  The new comparison is between this product of finite
tridiagonal sections and the \(2n\times2n\) Toeplitz matrix formed from their
product symbol.  Proposition~\ref{prop:finiteN_root} represents the resulting
determinant ratio by an \(m\times m\) determinant.  More explicitly, put
\[
\xi_h:=\cos\frac{h\pi}{2m+1},
\qquad
q_h:=\sqrt{1+\xi_h^2}-\xi_h
\quad (1\le h\le m).
\]
Then the tiling specialization is \(a_h=2\ii\xi_h\), the root in the open
unit disk is \(\rho_h=\ii q_h\), and, for fixed \(m\), the normalized ratio
\((-1)^{mn}Z_{m,n}/D_{2n}(F_m)\) tends to
\(\prod_{i<j}(1+q_iq_j)^2>0\), with the first exponentially small correction
given explicitly.  The precise statement is
Theorem~\ref{thm:domino_root_correction}.  The derivation from the dual Cauchy identity also connects
this calculation with the
general relation between Toeplitz minors---determinants of selected row and
column submatrices of a Toeplitz matrix---and skew Schur specializations
\cite{BD02,GGT20}, although the result proved here concerns ordinary
determinants.

\paragraph{Scope and contribution.}
The purpose of the paper is to work out explicitly the consequences of these
classical identities for the tridiagonal family with inversion symmetry and
its specialization to domino tilings.  The boundary, Schur, and root formulas
give complementary descriptions of the determinant comparison.  In the
domino specialization, the correction factor involves a determinant of order
\(m\), independent of \(n\), and its first exponentially small term is given in
\eqref{eq:domino_root_first}.

Section~\ref{sec:corner} derives the case of two factors from the product
formula for finite sections, proves the general identity for banded symbols by the method
of images, and then recovers it by iterating that formula.
Section~\ref{sec:determinants} derives the reduction to a determinant of fixed
dimension, its factorization under reversal,
and the exact root formula for finite \(N\).  Section~\ref{sec:domino} gives the
rectangular domino specialization.  Section~\ref{sec:outlook} discusses
growing bandwidths and Toeplitz minors.

\section{The Toeplitz identity with Hankel corrections}
\label{sec:corner}

We now prove the identity for finite bandwidth stated in the introduction.  In the
conventions of
\eqref{eq:TN_definition_intro}--\eqref{eq:HN_JN_definition_intro}, a Laurent
series \(G(z)=\sum_{k\in\mathbb Z}g_kz^k\) satisfies
\[
T_N(G)=[g_{i-j}]_{i,j=1}^N,
\qquad
D_N(G):=\det T_N(G),
\qquad
H_N(G)=[g_{i+j}]_{i,j=1}^N .
\]
The shift \(i+j\), rather than \(i+j-1\), is deliberate: it produces the
term at the upper left in \eqref{eq:ToeplitzHankel_intro}, while conjugation by
\(J_N\) produces the term at the lower right.
Throughout, the superscript $\top$ denotes ordinary rather than conjugate
transpose, even though the parameters and matrices need not be real.

For \(a\in\bbC\), the matrix \(T_N(z+z^{-1}+a)\) is the tridiagonal
Toeplitz matrix with diagonal entries \(a\), first upper and lower diagonals
equal to \(1\), and all other entries zero.  Equivalently,
\[
T_N(z+z^{-1}+a)=L_N+aI_N,
\qquad L_N:=T_N(z+z^{-1}).
\]
Because every matrix \(L_N+aI_N\) is a polynomial in \(L_N\), all matrices
in this family commute.  The walk expansion below naturally produces the
reverse product \(A_m\cdots A_1\), and commutativity allows us to write it as
\(A_1\cdots A_m\).

If \(v\in\bbC^N\) is extended by \(v_0=v_{N+1}=0\), then
\[
(L_Nv)_i=v_{i-1}+v_{i+1},
\qquad 1\le i\le N.
\]
These two zero endpoint values are what is meant by Dirichlet boundary
conditions in this discrete setting.

The identity compares a product of \(m\) tridiagonal Toeplitz matrices with
the Toeplitz matrix whose symbol is the product of their symbols.  For two
factors it follows directly from the product formula for finite sections.  We
give that derivation after the entrywise calculation below.  For an arbitrary
polynomial, a walk expansion makes the full image sum and the degree cutoff
transparent.  The discrepancy consists of walks that leave
\(\{1,\dots,N\}\), and the method of images removes them by subtracting
reflected kernels on the full integer line.

\paragraph{The case of two factors.}
Let $f_a(z)=z+z^{-1}+a$ and $f_b(z)=z+z^{-1}+b$, and set $F_2(z)=f_a(z)f_b(z)$.
Then \(T_N(f_a)\) and \(T_N(f_b)\) are tridiagonal Toeplitz matrices, while
\(T_N(F_2)\) can have nonzero entries only when \(|i-j|\le2\).
For \(N\ge2\), the next proposition shows that
\(T_N(f_a)T_N(f_b)\) and \(T_N(F_2)\) coincide except at the first and last
diagonal entries.  It also includes the scalar case \(N=1\).

\begin{proposition}[The identity for two factors]\label{prop:m2_corner}
For every $N\ge 1$ and $a,b\in\bbC$,
\begin{equation}\label{eq:m2_corner}
T_N(f_a)\,T_N(f_b)=T_N(F_2)-e_1e_1^\top-e_Ne_N^\top,
\end{equation}
where $e_1,e_N$ are the first and last standard basis vectors in $\bbC^N$.
Equivalently,
\begin{equation}\label{eq:m2_corner_hankel}
T_N(f_a)\,T_N(f_b)=T_N(F_2)-H_N(F_2)-J_NH_N(F_2)J_N,
\end{equation}
since $H_N(F_2)=e_1e_1^\top$ for all $N\ge 1$.
\end{proposition}

\begin{proof}
A direct computation shows that the product $T_N(f_a)T_N(f_b)$ has bandwidth~$2$ and, whenever the indices stay in range,
\[
(T_N(f_a)T_N(f_b))_{i,i\pm2}=1,\qquad (T_N(f_a)T_N(f_b))_{i,i\pm1}=a+b.
\]
For $N\ge2$, the diagonal entries are
\[
(T_N(f_a)T_N(f_b))_{i,i}=ab+2\quad (2\le i\le N-1),
\qquad
(T_N(f_a)T_N(f_b))_{1,1}=(T_N(f_a)T_N(f_b))_{N,N}=ab+1,
\]
the boundary cases having one fewer neighbor.
On the other hand,
\[
F_2(z)=f_a(z)f_b(z)=z^2+(a+b)z+(ab+2)+(a+b)z^{-1}+z^{-2},
\]
so $T_N(F_2)$ has the same five diagonals with diagonal coefficient $ab+2$.
Therefore \(T_N(f_a)T_N(f_b)\) and \(T_N(F_2)\) coincide except that the
\((1,1)\) and \((N,N)\) entries are smaller by \(1\).  This proves
\eqref{eq:m2_corner} when \(N\ge2\).  If \(N=1\), the product is the scalar
\(ab\), while \(T_1(F_2)=ab+2\).  The two corner corrections coincide and
the same identity follows.

For \eqref{eq:m2_corner_hankel}, the only Laurent coefficient of \(F_2\)
that can appear in \(H_N(F_2)=[c_{i+j}]_{i,j=1}^N\) is \(c_2=1\), because
\(i+j\ge2\) and \(c_k=0\) for \(k>2\).  Hence
\(H_N(F_2)=e_1e_1^\top\) and
\(J_NH_N(F_2)J_N=e_Ne_N^\top\).
\end{proof}

\begin{proof}[Second proof]
We use the formula for finite sections.  For a Laurent polynomial \(q\), write
\(q(z)=\sum_{k\in\mathbb Z}q_kz^k\).  Define the two semi-infinite Hankel
matrices
\[
\mathcal H_+(q):=[q_{r+s-1}]_{r,s\ge1},
\qquad
\mathcal H_-(q):=[q_{-(r+s-1)}]_{r,s\ge1}.
\]
Let \(\Pi_N\) be the orthogonal projection onto the first \(N\) coordinates,
and identify its range with \(\bbC^N\).  In the coefficient convention
\(T_N(q)=[q_{i-j}]_{i,j=1}^N\), Widom's product formula for finite sections
\cite[Eq.~(1.4)]{Widom76} reads
\begin{equation}\label{eq:widom_finite_section}
\begin{split}
T_N(uv)={}&T_N(u)T_N(v)
 +\Pi_N\mathcal H_+(u)\mathcal H_-(v)\Pi_N\\
&+J_N\bigl(\Pi_N\mathcal H_-(u)\mathcal H_+(v)\Pi_N\bigr)J_N .
\end{split}
\end{equation}
This is also Proposition~3.10 of \cite{BG05}.  Indeed, the entry in position
\((i,j)\) of
the first correction is
\(\sum_{\ell\le0}u_{i-\ell}v_{\ell-j}\), and the reversed correction gives the
corresponding sum over \(\ell\ge N+1\).  All the symbols here are Laurent
polynomials, so the products in \eqref{eq:widom_finite_section} have only
finitely many nonzero terms.  Specialize now to \(u=f_a\) and \(v=f_b\).
These symbols are invariant under \(z\mapsto z^{-1}\), and
\[
\mathcal H_+(f_a)=\mathcal H_-(f_a)
=\mathcal H_+(f_b)=\mathcal H_-(f_b)
=\varepsilon_1\varepsilon_1^\top,
\]
where \(\varepsilon_1\) is the first standard basis vector of
\(\ell^2(\mathbb Z_{\ge1})\).
The two correction terms in \eqref{eq:widom_finite_section} are therefore
\(e_1e_1^\top\) and \(e_Ne_N^\top\).  Rearranging gives
\eqref{eq:m2_corner}.  When \(N=1\), the two compressed corrections coincide
and both must be retained, giving the scalar correction \(2\).
\end{proof}

\paragraph{The case of three factors.}
The next case displays the nontrivial Hankel pattern.  Write
$P(x)=(x+a_1)(x+a_2)(x+a_3)$ and
$s_1=a_1+a_2+a_3$.  Since the positive coefficients of
$P(z+z^{-1})$ at degrees $2$ and $3$ are $s_1$ and $1$, respectively,
for $N\ge2$ one has
\[
H_N(P(z+z^{-1}))
=E_{+,2}
\begin{pmatrix}
s_1&1\\
1&0
\end{pmatrix}
E_{+,2}^\top,
\qquad
E_{+,2}:=[e_1\ e_2]\in\bbC^{N\times2},
\]
where \(E_{+,2}\) embeds \(\bbC^2\) into the first two coordinates of
\(\bbC^N\).
Thus, when \(m=3\), each boundary correction is supported in a
\(2\times2\) Hankel block.

For general \(m\), Theorem~\ref{thm:ToeplitzHankelCorners} gives the same
structure.  When \(N\ge m-1\), the two corner blocks have size
\((m-1)\times(m-1)\) and are determined by the Laurent coefficients of
\(F_m\).

\paragraph{Weighted walks and the method of images.}
Define the Laurent polynomial
\begin{equation}\label{eq:Fm_def}
F_m(z)=\prod_{h=1}^{m}\big(z+z^{-1}+a_h\big)=\sum_{k=-m}^{m}c_k z^k,
\qquad c_{-k}=c_k,\qquad c_k=0\ (|k|>m).
\end{equation}
The coefficients $c_k$ have a direct path interpretation:
\begin{equation}\label{eq:ck_paths}
c_k=\!\!\sum_{\substack{\varepsilon_1,\dots,\varepsilon_m\in\{-1,0,1\}\\ \varepsilon_1+\cdots+\varepsilon_m=k}}
\ \prod_{t=1}^{m} w_t(\varepsilon_t),
\qquad w_t(\pm 1)=1,\ \ w_t(0)=a_t.
\end{equation}

\begin{proposition}[Weighted walk representation of the product]\label{prop:walk_product}
Let \(m,N\ge1\) and \(a_1,\dots,a_m\in\bbC\).  For \(t=1,\dots,m\), set
\[
A_t:=T_N(z+z^{-1}+a_t),
\qquad
w_t(\pm1):=1,\quad w_t(0):=a_t.
\]
Then for every $i,j\in\{1,\dots,N\}$,
\begin{equation}\label{eq:walk_entry}
(A_1\cdots A_m)_{i,j}=\sum_{\substack{x_0=j,\dots,x_m=i\\ x_t\in\{1,\dots,N\}\\ x_t-x_{t-1}\in\{-1,0,1\}}}\ \prod_{t=1}^{m} w_t(x_t-x_{t-1}),
\end{equation}
i.e.\ $(A_1\cdots A_m)_{i,j}$ is the total weight of walks of length $m$ on the segment $\{1,\dots,N\}$ with step set $\{-1,0,1\}$ and weights $w_t$.
\end{proposition}

\begin{proof}
The displayed path sum is the direct expansion of $(A_m\cdots A_1)_{i,j}$:
\[
(A_m\cdots A_1)_{i,j}
=\sum_{x_1,\dots,x_{m-1}}(A_1)_{x_1,j}(A_2)_{x_2,x_1}\cdots (A_m)_{i,x_{m-1}}.
\]
The entry \((A_t)_{p,q}\) is \(1\) if \(p-q=\pm1\), is \(a_t\) if \(p=q\),
and is \(0\) otherwise.  Each summand is therefore the weight of a walk with
increments in \(\{-1,0,1\}\), and the condition
\(x_t\in\{1,\dots,N\}\) keeps the walk inside the segment.
Finally, $A_t=L_N+a_tI_N$, so the $A_t$ commute pairwise.  Hence
$A_m\cdots A_1=A_1\cdots A_m$.
\end{proof}

\smallskip
\noindent
If the walks in \eqref{eq:walk_entry} were allowed on the full integer line
\(\mathbb Z\), translation invariance would give the Toeplitz kernel
\(c_{i-j}\).  Restricting them to \(\{1,\dots,N\}\) removes paths that cross
either endpoint.  The following reflection formula computes the resulting
correction as an alternating sum of Laurent coefficients.  Such reflection
arguments for constrained lattice walks go back to the classical reflection
principle, in the form used for walks in Weyl chambers
\cite{GesselZeilberger92}.  \Cref{fig:images} illustrates the cancellation.

\begin{proposition}[Method of images for symmetric walks with weights depending on time]\label{prop:images}
With \(A_t\) as in Proposition~\ref{prop:walk_product} and \(c_k\) as in
\eqref{eq:Fm_def}, for every $m,N\ge 1$ and $i,j\in\{1,\dots,N\}$,
\begin{equation}\label{eq:images_formula}
(A_1\cdots A_m)_{i,j}
=\sum_{r\in\mathbb Z}\Big(c_{\,i-j+2r(N+1)}-c_{\,i+j+2r(N+1)}\Big),
\end{equation}
where $c_k$ are the coefficients of $F_m$ in \eqref{eq:Fm_def}.  The sum is
finite because $c_k=0$ for $|k|>m$.
\end{proposition}

\begin{proof}
Let \(\ell:=N+1\) and let
\(F_t(z):=\prod_{s=1}^{t}(z+z^{-1}+a_s)=\sum_k c_k^{(t)}z^k\), so
\(c_k^{(m)}=c_k\).
Multiplying by $(z+z^{-1}+a_t)$ shows the coefficients satisfy the recursion
\begin{equation}\label{eq:coeff_rec}
c_k^{(t)}=c_{k-1}^{(t-1)}+a_t\,c_k^{(t-1)}+c_{k+1}^{(t-1)},
\qquad c_k^{(0)}=\delta_{k,0}.
\end{equation}
Here \(\delta_{k,0}\) equals \(1\) when \(k=0\) and \(0\) otherwise.

Define the ``free'' kernel on the integer line by $K_t^{\mathrm{free}}(x,y):=c_{x-y}^{(t)}$ for $x,y\in\mathbb Z$.
By \eqref{eq:coeff_rec}, it solves the discrete evolution
\[
K_t^{\mathrm{free}}(x,y)=K_{t-1}^{\mathrm{free}}(x-1,y)+a_t K_{t-1}^{\mathrm{free}}(x,y)+K_{t-1}^{\mathrm{free}}(x+1,y).
\]

Now define, for \(i\in\{0,1,\dots,\ell\}\), the kernel obtained by summing over images
\begin{equation}\label{eq:Ktilde}
\widetilde K_t(i,j):=\sum_{r\in\mathbb Z}\Big(c_{\,i-j+2r\ell}^{(t)}-c_{\,i+j+2r\ell}^{(t)}\Big),
\qquad j\in\{1,\dots,N\}.
\end{equation}
The sum is finite for each fixed $t$ because $c_k^{(t)}=0$ for $|k|>t$.

\smallskip
\noindent\emph{(i) Boundary conditions.}
Using the palindromic symmetry $c_{-k}^{(t)}=c_k^{(t)}$ and the fact that the sum runs over all $r\in\mathbb Z$,
\[
\widetilde K_t(0,j)=\sum_{r}\big(c_{-j+2r\ell}^{(t)}-c_{j+2r\ell}^{(t)}\big)=0,
\qquad
\widetilde K_t(\ell,j)=\sum_{r}\big(c_{\ell-j+2r\ell}^{(t)}-c_{\ell+j+2r\ell}^{(t)}\big)=0,
\]
(the second identity follows by applying $c_{-k}^{(t)}=c_k^{(t)}$ and then changing variables $r\mapsto -r-1$ in the first sum).

\smallskip
\noindent\emph{(ii) Evolution in $t$.}
For $i\in\{1,\dots,N\}$, applying \eqref{eq:coeff_rec} termwise in \eqref{eq:Ktilde} gives
\[
\widetilde K_t(i,j)=\widetilde K_{t-1}(i-1,j)+a_t\,\widetilde K_{t-1}(i,j)+\widetilde K_{t-1}(i+1,j),
\]
with boundary values
\(\widetilde K_t(0,j)=\widetilde K_t(\ell,j)=0\) from (i).  At \(t=0\),
only the term \(r=0\) can contribute to \eqref{eq:Ktilde}, and hence
\(\widetilde K_0(i,j)=\delta_{i,j}\) for \(1\le i,j\le N\).

\smallskip
\noindent\emph{(iii) Identification with the matrix product.}
For fixed \(j\), the vector
\(v_t(i):=(A_t\cdots A_1)_{i,j}\), \(1\le i\le N\), satisfies the same
recursion in \(t\), with the same boundary and initial data.  This recursion is
left multiplication by \(A_t\).  Induction gives
\(v_t(i)=\widetilde K_t(i,j)\).  Since the matrices
\(A_t=L_N+a_tI_N\) commute pairwise,
\(A_m\cdots A_1=A_1\cdots A_m\), and \eqref{eq:images_formula} follows at
\(t=m\).
\end{proof}

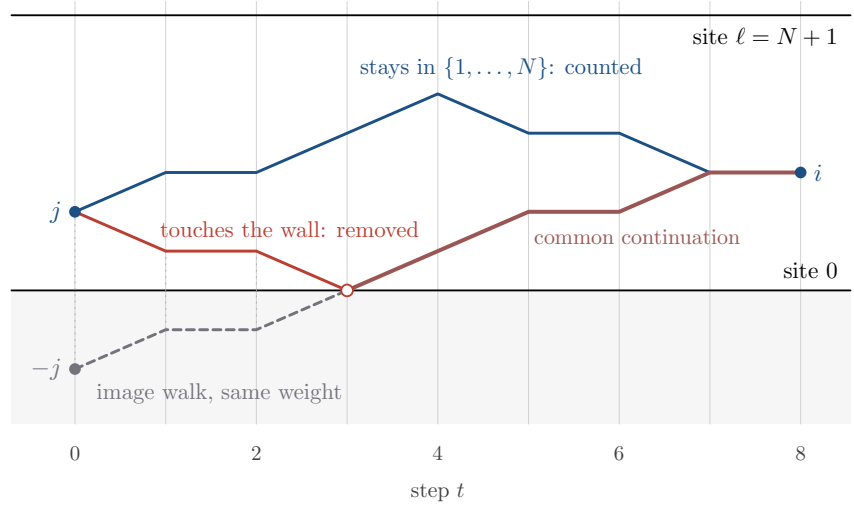
\begin{figure}[H]
\centering
\begin{tikzpicture}[x=12mm,y=5.2mm,line join=round,line cap=round]
  \def\T{8}
  \fill[gray!7] (-0.7,-3.4) rectangle (\T+0.55,0);
  \foreach \t in {0,...,8}{\draw[gray!35,line width=.15pt] (\t,-3.4) -- (\t,7.35);}
  \draw[black,line width=.7pt] (-0.7,0) -- (\T+0.55,0);
  \draw[black,line width=.7pt] (-0.7,7) -- (\T+0.55,7);
  \node[scale=.8,anchor=south east] at (\T+0.5,0.1) {site $0$};
  \node[scale=.8,anchor=north east] at (\T+0.5,6.9) {site $\ell=N+1$};
  \foreach \t in {0,2,4,6,8}{\node[scale=.72,text=gray!50!black] at (\t,-4.15) {$\t$};}
  \node[scale=.75,text=gray!50!black] at (4,-5.15) {step $t$};
  \draw[accent,line width=1.1pt]
      plot coordinates {(0,2)(1,3)(2,3)(3,4)(4,5)(5,4)(6,4)(7,3)(8,3)};
  \foreach \t/\ya/\yb in {0/2/-2,1/1/-1,2/1/-1}{
     \draw[gray!60,line width=.4pt,dotted] (\t,\ya) -- (\t,\yb);}
  \draw[imgcol,line width=1.1pt,densely dashed]
      plot coordinates {(0,-2)(1,-1)(2,-1)(3,0)};
  \draw[cornercol,line width=1.1pt]
      plot coordinates {(0,2)(1,1)(2,1)(3,0)};
  \draw[cornercol!55!imgcol,line width=1.5pt]
      plot coordinates {(3,0)(4,1)(5,2)(6,2)(7,3)(8,3)};
  \fill[accent] (0,2) circle (2.2pt);
  \node[scale=.85,text=accent,left=2pt] at (0,2) {$j$};
  \fill[imgcol] (0,-2) circle (2.2pt);
  \node[scale=.85,text=imgcol,left=2pt] at (0,-2) {$-j$};
  \fill[accent] (8,3) circle (2.2pt);
  \node[scale=.85,text=accent,right=2pt] at (8,3) {$i$};
  \draw[cornercol,fill=white,line width=.7pt] (3,0) circle (2.2pt);
  \node[scale=.78,text=accent,anchor=south] at (4.7,5.15)
       {stays in $\{1,\dots,N\}$: counted};
  \node[scale=.78,text=cornercol,anchor=south west] at (0.85,1.15)
       {touches the wall: removed};
  \node[scale=.78,text=imgcol,anchor=north west] at (0.15,-2.15)
       {image walk, same weight};
  \node[scale=.75,text=cornercol!55!imgcol,anchor=north] at (6.2,1.75)
       {common continuation};
\end{tikzpicture}

\smallskip
{\small$\displaystyle
 (A_1\cdots A_m)_{i,j}=
 \underbrace{c_{i-j}}_{\text{free walks}}
 \;-\;\underbrace{c_{i+j}}_{\text{image in }0}
 \;-\;\underbrace{c_{2(N+1)-i-j}}_{\text{image in }\ell}$}

\caption{The method of images behind Proposition~\ref{prop:images} and
\Cref{thm:ToeplitzHankelCorners}.  Steps lie in $\{-1,0,1\}$ and
carry the weights $w_t(\pm1)=1$, $w_t(0)=a_t$.  The walk in blue never leaves
$\{1,\dots,N\}$ and is counted by the matrix product.  The walk in red starts
at $j$ but touches the wall at site $0$, so it must be removed.  Reflecting its
initial segment in that wall (dotted) matches it bijectively, and with equal
weight, to a walk started at the image site $-j$.  Subtracting all walks from
$-j$ therefore removes exactly the offending walks, and contributes the Hankel
kernel $c_{i+j}$.  The wall at $\ell=N+1$ contributes $c_{2(N+1)-i-j}$ in the
same way.  All other terms in \eqref{eq:images_formula} have coefficient indices
of absolute value at least $N+3$.  They therefore vanish when $m\le N+2$,
leaving the three displayed terms.}
\label{fig:images}
\end{figure}

\paragraph{From images to the matrix identity.}
The following identity with Hankel corrections is stated for an
arbitrary polynomial.  The factorized symbol used later is its monic special
case.

\begin{theorem}[Toeplitz identity with Hankel corrections]\label{thm:ToeplitzHankelCorners}
Let \(N\ge1\), let $P\in\bbC[x]$ have degree $d\ge1$, and set
\[
G(z):=P(z+z^{-1}),
\qquad L_N:=T_N(z+z^{-1}),
\]
and assume $d\le N+2$.  Then
\begin{equation}\label{eq:ToeplitzHankelIdentity}
P(L_N)=T_N(G)-H_N(G)-J_NH_N(G)J_N.
\end{equation}
If $N\ge d-1$, the two Hankel terms are supported in the upper left and
lower right $(d-1)\times(d-1)$ corners, respectively.

In particular, taking $P(x)=\prod_{h=1}^m(x+a_h)$ gives
\[
\prod_{h=1}^{m}T_N(z+z^{-1}+a_h)
=T_N(F_m)-H_N(F_m)-J_NH_N(F_m)J_N .
\]
\end{theorem}

\begin{proof}
Factor
\[
P(x)=p_d\prod_{h=1}^d(x+a_h)
\]
over \(\bbC\).  Factoring out the leading coefficient \(p_d\), it is enough to
prove the monic case.  Set
\[
G(z)=\prod_{h=1}^d(z+z^{-1}+a_h)
=\sum_{k=-d}^d c_kz^k .
\]
Propositions~\ref{prop:walk_product} and \ref{prop:images} apply with \(m=d\).
Multiplication by \(p_d\) at the end gives the general case.

Fix \(i,j\in\{1,\dots,N\}\).  In Proposition~\ref{prop:images}, the term
with \(r=0\) equals \(c_{i-j}-c_{i+j}\).  For \(r=1\),
\[
|i-j+2(N+1)|\ge2(N+1)-(N-1)=N+3>d,
\]
and \(i+j+2(N+1)>d\), so both coefficients vanish.  For \(r\le-2\) or
\(r\ge2\), the relevant indices have absolute value at least \(N+3\), and
their coefficients vanish as well.
Thus only $r=0$ and $r=-1$ can contribute.  The $r=-1$ contribution is
\(
c_{i-j-2(N+1)}-c_{i+j-2(N+1)}.
\)
The first term vanishes because \(|i-j-2(N+1)|\ge N+3>d\).
The second term is $-c_{i+j-2(N+1)}$, which equals $-c_{2(N+1)-i-j}$ by the palindromic symmetry $c_{-k}=c_k$.
Therefore \eqref{eq:images_formula} reduces to
\[
(A_1\cdots A_d)_{i,j}=c_{i-j}-c_{i+j}-c_{2(N+1)-i-j},
\]
which is exactly the $(i,j)$ entry of the right side of \eqref{eq:ToeplitzHankelIdentity}.
If \(N\ge d-1\), then \(c_{i+j}=0\) whenever \(i+j>d\).  Thus every
nonzero entry of \(H_N(G)\) has \(i,j\le d-1\).  Conjugation by \(J_N\) gives
the corresponding support at the lower right.
\end{proof}

\begin{remark}[Iteration of the product formula for finite sections and sharpness of the degree bound]
\label{rem:degree_sharp}
Iterating \eqref{eq:widom_finite_section} also gives a second proof of
\Cref{thm:ToeplitzHankelCorners} and identifies the additional boundary terms
that first appear when the degree bound fails.  It is enough to consider the
monic factorization
\[
Q_t(x):=\prod_{h=1}^t(x+a_h),
\qquad
G_t(z):=Q_t(z+z^{-1})=\sum_k c_k^{(t)}z^k .
\]
Set \(A_t:=T_N(z+z^{-1}+a_t)\) and \(H_t:=H_N(G_t)\).  Since
\(c_{-k}^{(t-1)}=c_k^{(t-1)}\), applying \eqref{eq:widom_finite_section} to
\(G_{t-1}\) and \(z+z^{-1}+a_t\) gives
\begin{equation}\label{eq:widom_iteration}
T_N(G_t)=T_N(G_{t-1})A_t+U_t+J_NU_tJ_N,
\qquad
U_t=u^{(t-1)}e_1^\top,
\end{equation}
where
\[
u^{(t-1)}:=
(c_1^{(t-1)},\ldots,c_N^{(t-1)})^\top .
\]
Define also
\[
v^{(t-1)}:=
(c_{N+2}^{(t-1)},c_{N+3}^{(t-1)},\ldots,c_{2N+1}^{(t-1)})^\top .
\]
The coefficient recurrence
\[
c_k^{(t)}=c_{k-1}^{(t-1)}+a_t c_k^{(t-1)}+c_{k+1}^{(t-1)}
\]
gives, by direct matrix multiplication,
\begin{equation}\label{eq:widom_hankel_iteration}
U_t+H_{t-1}A_t=H_t-v^{(t-1)}e_N^\top .
\end{equation}
Relative to \(H_t\), the matrix \(U_t+H_{t-1}A_t\) lacks precisely the entries
in \(v^{(t-1)}\) in its last column.  Since
\(J_NA_tJ_N=A_t\), conjugating \eqref{eq:widom_hankel_iteration} by \(J_N\)
gives
\[
J_NU_tJ_N+(J_NH_{t-1}J_N)A_t
=J_NH_tJ_N-(J_Nv^{(t-1)})e_1^\top .
\]
For \(t=1\), one has \(Q_1(L_N)=T_N(G_1)\) and \(H_1=0\).  If
\eqref{eq:ToeplitzHankelIdentity} holds for \(Q_{t-1}\), then
\eqref{eq:widom_iteration}--\eqref{eq:widom_hankel_iteration} give
\begin{equation}\label{eq:widom_remainder}
Q_t(L_N)=T_N(G_t)-H_t-J_NH_tJ_N
 +v^{(t-1)}e_N^\top+(J_Nv^{(t-1)})e_1^\top .
\end{equation}
If \(t\le N+2\), then \(G_{t-1}\) has Laurent degree at most \(N+1\), while
every coefficient in \(v^{(t-1)}\) has index at least \(N+2\).  Hence
\(v^{(t-1)}=0\), and induction proves
\eqref{eq:ToeplitzHankelIdentity} for every \(t\le N+2\).  Multiplication by
the leading coefficient gives the result for an arbitrary polynomial.

At the first excluded degree, \(t=N+3\), the identity at step \(t-1=N+2\)
still holds.  Because \(G_{N+2}\) is monic of Laurent degree \(N+2\), one has
\(v^{(N+2)}=e_1\).  Formula \eqref{eq:widom_remainder} therefore becomes
\[
Q_{N+3}(L_N)=T_N(G_{N+3})-H_N(G_{N+3})
-J_NH_N(G_{N+3})J_N+e_1e_N^\top+e_Ne_1^\top .
\]
Thus the degree condition cannot be increased uniformly.  For \(N=1\), the
two residual matrices coincide and their sum is the scalar \(2\), in
agreement with the example \(Q_4(x)=x^4\).
\end{remark}

\section{Determinant ratios and boundary corrections}\label{sec:determinants}

For fixed \(m\), the product \(P(L_N)\) is diagonalized by the discrete sine
transform, while its difference from \(T_N(F_m)\) has rank bounded
independently of \(N\).  We first represent the determinant ratio by a matrix
of fixed dimension, then evaluate its limit in terms of the roots \(\rho_h\),
and finally derive
an exact formula valid for every \(N\) from the
Borodin--Okounkov--Geronimo--Case identity.

\subsection{Reduction to finite dimension}
When $N\ge m-1$, the two Hankel corrections in Theorem~\ref{thm:ToeplitzHankelCorners} are supported in the $(m-1)\times(m-1)$ corners at the upper left and lower right.
In particular, the difference between $\prod_{h=1}^m T_N(f_h)$ and
$T_N(F_m)$ has rank at most $2(m-1)$, uniformly in \(N\).
This makes the determinant ratio computable by the matrix determinant lemma
on a space of dimension \(2(m-1)\).

For \(m=1\), both Hankel corrections vanish and the determinant ratio is
\(1\) whenever defined.  In the remainder of this subsection, assume \(m\ge2\).

\begin{proposition}[Determinant reduction for finite $N$]\label{prop:finiteN_reduction}
Assume \(m\ge2\) and $N\ge m-1$, and set $k:=m-1$.
Let $E_+\in\bbC^{N\times k}$ be the embedding
$E_+=[I_k\ 0]^\top$ selecting the first $k$ coordinates, and set
$E_-:=J_NE_+$.
Then
\[
H_N(F_m)=E_+H_k(F_m)E_+^\top,\qquad
J_NH_N(F_m)J_N=E_-H_k(F_m)E_-^\top,
\]
so
\begin{align}
\prod_{h=1}^{m}T_{N}(z+z^{-1}+a_h)
&= T_N(F_m)+U\Gamma U^\top,\label{eq:finite_rank_form}\\
U&:=[E_+\ E_-]\in\bbC^{N\times 2k},\qquad
\Gamma:=\operatorname{diag}\big(-H_k(F_m),-H_k(F_m)\big).
\end{align}
If $T_N(F_m)$ is invertible, then
\begin{equation}\label{eq:det_reduction}
\det\!\Big(\prod_{h=1}^{m}T_{N}(z+z^{-1}+a_h)\Big)
=\det T_N(F_m)\ \det\!\Big(I_{2k}+\Gamma\,U^\top T_N(F_m)^{-1}U\Big),
\end{equation}
so the ratio of determinants is a finite $2k\times2k$ determinant. Put
\[
U^\top T_N(F_m)^{-1}U=
\begin{pmatrix}
M_{++}&M_{+-}\\
M_{-+}&M_{--}
\end{pmatrix},
\qquad
M_{\sigma\tau}:=E_\sigma^\top T_N(F_m)^{-1}E_\tau
\quad(\sigma,\tau\in\{+,-\}).
\]
Then \eqref{eq:det_reduction} becomes
\begin{equation}\label{eq:block_ratio}
\frac{\det(\prod_{h=1}^{m}T_{N}(z+z^{-1}+a_h))}{\det T_N(F_m)}
=
\det\!\begin{pmatrix}
I_k-H_k(F_m)M_{++} & -H_k(F_m)M_{+-}\\
-H_k(F_m)M_{-+} & I_k-H_k(F_m)M_{--}
\end{pmatrix}.
\end{equation}
\end{proposition}

\begin{corollary}[Factorization by reversal symmetry]
\label{cor:parity_factorization}
Under the hypotheses of Proposition~\ref{prop:finiteN_reduction}, suppose
$T_N(F_m)$ is invertible.  With the blocks from
\eqref{eq:block_ratio}, one has
\[
M_{--}=M_{++},\qquad M_{-+}=M_{+-},
\]
and hence
\begin{equation}\label{eq:parity_factorization}
R_N=
\det\!\bigl(I_k-H_k(F_m)(M_{++}+M_{+-})\bigr)\,
\det\!\bigl(I_k-H_k(F_m)(M_{++}-M_{+-})\bigr),
\end{equation}
where \(R_N\) is the determinant ratio in \eqref{eq:R_N_intro}.
\end{corollary}

\begin{proof}
Because $F_m(z)=F_m(z^{-1})$, the matrix $T_N(F_m)$ is symmetric and
\emph{centrosymmetric}, meaning that
\(J_NT_N(F_m)J_N=T_N(F_m)\).  Its inverse therefore has the same two
properties and commutes with $J_N$.  Since $E_-=J_NE_+$, this gives
\[
M_{--}=E_+^\top J_NT_N(F_m)^{-1}J_NE_+=M_{++},
\qquad
M_{-+}=E_+^\top J_NT_N(F_m)^{-1}E_+=M_{+-}.
\]
The block determinant in \eqref{eq:block_ratio} consequently has the form
$\bigl(\begin{smallmatrix}A&B\\ B&A\end{smallmatrix}\bigr)$.
The subspaces consisting of vectors \((v,v)\) and \((v,-v)\) are invariant,
and the restrictions to them are \(A+B\) and \(A-B\), respectively.  This
gives \(\operatorname{diag}(A+B,A-B)\) after a change of basis and proves
\eqref{eq:parity_factorization}.  No commutativity among the $k\times k$
blocks is required.

These two subspaces are the \(+1\) and \(-1\) eigenspaces of the swap
\begin{equation}\label{eq:swap}
\mathcal R:=\begin{pmatrix}0&I_k\\I_k&0\end{pmatrix}\in\bbC^{2k\times2k}.
\end{equation}
Since \(E_-=J_NE_+\) and \(J_N^2=I_N\), one has \(U\mathcal R=J_NU\), so
\(\mathcal R\) represents reversal on the domain
\(\bbC^k\oplus\bbC^k\) of \(U\).  The factorization is thus always into two
eigenspaces of dimension \(k\), whatever
the relative size of \(N\) and \(2k\).  When \(N<2k\) the map \(U\) is not
injective and \(\mathcal R\) is not induced by a splitting of \(\bbC^N\)
itself, but \eqref{eq:parity_factorization} is unaffected.
\end{proof}

\begin{remark}[Reflective symmetry]\label{rem:reflective}
The decomposition of a symmetric centrosymmetric matrix into subspaces that
are even and odd under reversal is classical
\cite{CantoniButler76}.  A related splitting approach for centrosymmetric
sums of Toeplitz and Hankel matrices appears in \cite{Heinig01}, and the same parity
split is used for the boundary systems associated with a banded Toeplitz
matrix in \cite[Eq.~(9)]{HendrickxVanBarel02}.  A different but
related factorization of a full Toeplitz determinant with even symbol into
two determinants, each involving the sum of a Toeplitz matrix and a Hankel
matrix, is given in
\cite[Thms.~5.1--5.2]{BasorConrey24}.  The dimensions of those
two factors grow with the size of the Toeplitz matrix.  In
\eqref{eq:parity_factorization}, both factors have the fixed dimension
\(m-1\).

Splitting \eqref{eq:block_ratio} into the \(+1\) and \(-1\) eigenspaces of
\eqref{eq:swap} is analogous to the factorization of matching generating
functions under reflection symmetry \cite{Ciucu97}.  Here the symmetry acts
on the reduced boundary determinant.
\end{remark}

\begin{remark}[Corner overlap]
The hypothesis $N\ge m-1$ is the condition needed to realize the Hankel block
at the upper left as an $N\times k$ embedding with $k=m-1$.  This also implies
\(m\le N+2\), the degree condition in
Theorem~\ref{thm:ToeplitzHankelCorners}.  When
$N<2(m-1)$, the two corner blocks overlap.  The identity and determinant
reduction remain valid without disjoint support.  In that
regime \(U\) is not injective, and the factorization of
Corollary~\ref{cor:parity_factorization} is by the eigenspaces of the swap
\eqref{eq:swap} rather than by a splitting of \(\bbC^N\).
\end{remark}

\begin{lemma}[Limit of the boundary blocks]\label{lem:finite_section_limit}
If \(F_m(z)=\sum_kc_kz^k\), let
\[
T(F_m):=[c_{i-j}]_{i,j\ge0}
\]
denote the semi-infinite Toeplitz operator on
\(\ell^2(\mathbb Z_{\ge0})\).  Suppose that \(F_m\) has no zeros on
\(\bbT\) and has winding number zero, meaning that the closed curve
\(F_m(\bbT)\) winds zero times around the origin.  Then \(T(F_m)\) and
\(T_N(F_m)\) for all sufficiently large \(N\) are invertible.  For fixed
\(k=m-1\), let \(\iota_k\) be the coordinate embedding
\[
\begin{aligned}
\iota_k:\bbC^k&\longrightarrow\ell^2(\mathbb Z_{\ge0}),\\
(x_1,\dots,x_k)&\longmapsto(x_1,\dots,x_k,0,0,\dots),
\end{aligned}
\]
and set
\[
S_{11}:=\iota_k^\top T(F_m)^{-1}\iota_k.
\]
With \(E_\pm\) as in Proposition~\ref{prop:finiteN_reduction},
\[
E_+^\top T_N(F_m)^{-1}E_+\longrightarrow S_{11},
\qquad
E_+^\top T_N(F_m)^{-1}E_-\longrightarrow0.
\]
Consequently, the determinant ratio in \eqref{eq:R_N_intro} converges to
\begin{equation}\label{eq:boundary_det_limit}
\det(I_k-H_k(F_m)S_{11})^2 .
\end{equation}
\end{lemma}

\begin{proof}
The invertibility and stability of the finite sections, together with
convergence of each fixed corner of the inverses, follow from the standard
theory of finite sections for scalar Toeplitz operators whose symbols are
nonvanishing and have winding number zero \cite{BS99}.  For related analyses
of perturbations in the corners,
see \cite{BFGM15}.

To prove that \(E_+^\top T_N(F_m)^{-1}E_-\to0\), a bound on the norm of the
inverse is not by itself enough, and we use instead exponential decay away
from the diagonal.
Indeed, put \(A_N:=T_N(F_m)\), and let \(A_N^*\) denote its Hermitian adjoint.
Stability of the finite sections gives a uniform bound on \(\|A_N^{-1}\|\) for all
sufficiently large \(N\), while
\(\|A_N\|\le\sum_{k=-m}^{m}|c_k|\) uniformly in \(N\).  Hence
\(A_N^*A_N\) is positive definite, satisfies
\((A_N^*A_N)_{ij}=0\) whenever \(|i-j|>2m\), and its spectral condition
number is bounded uniformly in \(N\).  The decay estimate for the inverse from
\cite{DemkoMossSmith84,BG05} applies to \((A_N^*A_N)^{-1}\).  Combining it with
\(A_N^{-1}=(A_N^*A_N)^{-1}A_N^*\) and the bandwidth of \(A_N^*\) yields
constants \(C>0\) and \(q\in(0,1)\), both independent of \(N\), such that
\begin{equation}\label{eq:inverse_decay}
\bigl|\bigl(T_N(F_m)^{-1}\bigr)_{ij}\bigr|\le Cq^{|i-j|},
\qquad 1\le i,j\le N,
\end{equation}
for all sufficiently large \(N\).  Only the existence of some such \(q\) is
needed below.  The entry in position \((i,j)\) of
\(E_+^\top T_N(F_m)^{-1}E_-\) is
\[
\bigl(T_N(F_m)^{-1}\bigr)_{i,N+1-j},\qquad 1\le i,j\le k,
\]
and
\[
\bigl|i-(N+1-j)\bigr|=N+1-i-j\ge N-2k+1.
\]
Hence
\[
\bigl\|E_+^\top T_N(F_m)^{-1}E_-\bigr\|\le Ckq^{\,N-2k+1}\longrightarrow0
\qquad(N\to\infty),
\]
with \(k\) fixed.  Substitution of these limits into
Proposition~\ref{prop:finiteN_reduction} gives
\eqref{eq:boundary_det_limit}.  The same limit follows by applying these block
limits separately to the two factors in \eqref{eq:parity_factorization}.
\end{proof}

\subsection{The limiting determinant ratio}
With \(k=m-1\), Lemma~\ref{lem:finite_section_limit} shows that the ratio
\(R_N\) in \eqref{eq:R_N_intro} converges to
\(\det(I_k-H_k(F_m)S_{11})^2\).  We evaluate this limit explicitly in terms
of the roots of the quadratic factors of \(F_m\).

Assume from now on that each
\[
f_h(z)=z+z^{-1}+a_h
\]
is nonzero on $\bbT$.  This is equivalent to the quadratic polynomial
\(z^2+a_hz+1\) having no zero on the unit circle, and also to
\(a_h\notin[-2,2]\subset\mathbb R\).  Since the two roots of this quadratic
have product \(1\) and neither lies on \(\bbT\), exactly one lies in the open
unit disk.  Denote it by \(\rho_h\).  Thus
\begin{equation}\label{eq:rho_def}
\rho_h^2+a_h\rho_h+1=0,
\qquad |\rho_h|<1.
\end{equation}
Equivalently, choose the branch of the square root for which
\begin{equation}\label{eq:rho_branch}
\rho_h:=\frac{-a_h+\sqrt{a_h^2-4}}{2}
\qquad\text{satisfies}\qquad |\rho_h|<1,
\end{equation}
so that the other root is $\rho_h^{-1}$.
Equation~\eqref{eq:rho_def} gives the factorization
\begin{equation}\label{eq:fh_factor}
f_h(z)=z+z^{-1}+a_h
=z^{-1}(z^2+a_h z+1)
=-\rho_h^{-1}\,(1-\rho_h z)\,(1-\rho_h z^{-1}).
\end{equation}

\begin{theorem}[Limit in terms of the roots in the open unit disk]\label{thm:corner_constant_roots}
Let \(m\ge1\), let \(a_1,\dots,a_m\in\bbC\), and suppose that
\(f_h(z)=z+z^{-1}+a_h\) is nonzero for every \(z\in\bbT\) and every
\(h\).  For each \(h\), let \(\rho_h\) be the unique root of
\(z^2+a_hz+1\) satisfying \(|\rho_h|<1\), and put
\(F_m=\prod_{h=1}^mf_h\).  Then \(D_N(F_m)\ne0\) for all sufficiently
large \(N\) and, for fixed \(m\) and \(N\to\infty\),
\begin{equation}\label{eq:corner_constant_product}
\frac{\det\big(\prod_{h=1}^m T_N(f_h)\big)}{\det T_N(F_m)}
\longrightarrow
\prod_{1\le i<j\le m}(1-\rho_i\rho_j)^2.
\end{equation}
If \(m\ge2\), then, with \(S_{11}\) defined in
Lemma~\ref{lem:finite_section_limit},
\begin{equation}\label{eq:detIk_explicit}
\det(I_{m-1}-H_{m-1}(F_m)S_{11})
=\prod_{1\le i<j\le m}(1-\rho_i\rho_j).
\end{equation}
For \(m=1\), the corresponding empty determinant and empty product are both
interpreted as \(1\).
\end{theorem}

\begin{proof}
Since determinants are multiplicative,
\(
\det\big(\prod_{h=1}^m T_N(f_h)\big)=\prod_{h=1}^m \det T_N(f_h).
\)
For a sufficiently regular, nonvanishing symbol $g$ of winding number zero,
the strong Szeg\H{o} theorem \cite{Szego52} gives (see, for instance,
\cite{BS99} or the survey \cite{DIK13})
\begin{equation}\label{eq:szego}
D_N(g)=
\exp\!\left(
N(\log g)_0+\sum_{k\ge1}k(\log g)_k(\log g)_{-k}
\right)(1+o(1)).
\end{equation}
Here zero winding guarantees a continuous branch of \(\log g\) on
\(\bbT\), and our convention for Fourier coefficients is
\[
(\log g)_k
:=\frac{1}{2\pi}\int_0^{2\pi}
\log g(e^{\ii\theta})e^{-\ii k\theta}\,d\theta.
\]
Regarded as a meromorphic function, each \(f_h\) has one zero and one pole in
\(|z|<1\), and hence has winding number zero on \(\bbT\).  The same is true of
\(F_m\).
Choose the branch for \(F_m\) to be
\(\log F_m=\sum_h\log f_h\).  Applying \eqref{eq:szego} to each \(f_h\)
and to \(F_m=\prod_h f_h\), the terms \(N(\log g)_0\) then cancel in the
ratio because
\((\log F_m)_0=\sum_h(\log f_h)_0\).

Using \eqref{eq:fh_factor} with $|\rho_h|<1$, we have on $\bbT$ the Fourier expansions
\(
\log(1-\rho_h z)=-\sum_{k\ge 1}\rho_h^k z^k/k,
\)
and
\(
\log(1-\rho_h z^{-1})=-\sum_{k\ge 1}\rho_h^k z^{-k}/k.
\)
Hence for $k\ge 1$,
\begin{equation}\label{eq:logfh_coeffs}
(\log f_h)_k=(\log f_h)_{-k}=-\frac{\rho_h^k}{k}.
\end{equation}
Therefore the Szeg\H{o} constant for $f_h$ is
\begin{equation}\label{eq:Ef_h}
\exp\Big(\sum_{k\ge 1}k(\log f_h)_k(\log f_h)_{-k}\Big)
=\exp\Big(\sum_{k\ge 1}\frac{\rho_h^{2k}}{k}\Big)=\frac{1}{1-\rho_h^2}.
\end{equation}

For $F_m$, we have $(\log F_m)_k=\sum_h (\log f_h)_k=-(\sum_h \rho_h^k)/k$ and similarly for $-k$, so
\begin{align}
\exp\Big(\sum_{k\ge 1}k(\log F_m)_k(\log F_m)_{-k}\Big)
&=\exp\Big(\sum_{k\ge 1}\frac{(\sum_h \rho_h^k)^2}{k}\Big)
=\exp\Big(\sum_{i,j=1}^m\sum_{k\ge 1}\frac{(\rho_i\rho_j)^k}{k}\Big)\\
&=\prod_{i,j=1}^m\frac{1}{1-\rho_i\rho_j}.
\end{align}
Combining this with \eqref{eq:Ef_h} and the cancellation of the $N(\log g)_0$ terms in \eqref{eq:szego} yields
\[
\lim_{N\to\infty}\frac{\det\big(\prod_h T_N(f_h)\big)}{\det T_N(F_m)}
=\frac{\prod_{h=1}^m (1-\rho_h^2)^{-1}}{\prod_{i,j}(1-\rho_i\rho_j)^{-1}}
=\prod_{i\ne j}(1-\rho_i\rho_j)
=\prod_{i<j}(1-\rho_i\rho_j)^2,
\]
which is \eqref{eq:corner_constant_product}.
The leading term supplied by \eqref{eq:szego} for $D_N(F_m)$ is nonzero,
so $D_N(F_m)\ne0$ for all sufficiently large $N$.
Comparing this limit with \eqref{eq:boundary_det_limit} gives the square of
\eqref{eq:detIk_explicit}.  It remains to identify the factor itself.

Multiplying a symbol by a nonzero scalar does not change the product
\(H_{m-1}(F_m)S_{11}\).  We may therefore replace \(F_m\) by
\[
g_{\boldsymbol\rho}(z)
:=\prod_{h=1}^m(1-\rho_hz)(1-\rho_hz^{-1}).
\]
For \(0<t\le1\), replace every \(\rho_h\) by \(t\rho_h\), and let
\(S_{11}(t)\) denote the corresponding block.  As
\(t\to0\), the symbol \(g_{t\boldsymbol\rho}\) converges to \(1\), so
\(H_{m-1}(g_{t\boldsymbol\rho})\to0\).  Norm continuity of inversion gives
\(T(g_{t\boldsymbol\rho})^{-1}\to I\), and hence
\(S_{11}(t)\to I_{m-1}\).  Therefore the determinant on the left of
\eqref{eq:detIk_explicit}, evaluated at \(t\boldsymbol\rho\), tends to \(1\),
as does its right hand side.  The product on the right is nonzero on the
connected set
\[
\bigl\{(\rho_1,\dots,\rho_m):0<|\rho_h|<1\bigr\}
\]
so the quotient of the two sides is continuous there.  Its square equals
\(1\), and it is therefore identically \(1\), which proves
\eqref{eq:detIk_explicit}.
\end{proof}

\begin{remark}[The case \(m=2\)]
Here \(H_1(F_2)=[1]\), so \eqref{eq:detIk_explicit} gives
\(S_{11}=\rho_1\rho_2\) directly.
\end{remark}

\subsection{The exact correction for finite \texorpdfstring{$N$}{N}}
For \(m\ge2\) and \(N\ge m-1\), Proposition~\ref{prop:finiteN_reduction} already represents \(R_N\) by a
\(2(m-1)\times2(m-1)\) determinant.  For \(m=1\), \(R_N=1\) identically.
We first record a specialization of the Toeplitz minor identity in
\cite[Eq.~(19)]{GGT20}.  It expresses \(D_N(F_m)\) in terms of the Schur
polynomial indexed by the rectangular partition
\((N^m)=(N,\dots,N)\) with \(m\) parts.  The parameters \(\rho_h\) need only
be nonzero.  They may coincide and need not lie in the open unit disk.

\begin{proposition}[Schur polynomial representation]\label{prop:rectangular_schur}
Let \(N,m\ge1\), let \(\rho_1,\dots,\rho_m\in\bbC\setminus\{0\}\), and write
\[
\boldsymbol\rho=(\rho_1,\dots,\rho_m),\qquad
\boldsymbol\rho^{-1}=(\rho_1^{-1},\dots,\rho_m^{-1}).
\]
Define
\[
a_h:=-(\rho_h+\rho_h^{-1}),\qquad
f_h(z):=z+z^{-1}+a_h,\qquad
F_m(z):=\prod_{h=1}^m f_h(z).
\]
Then
\begin{equation}\label{eq:rectangular_schur}
D_N(F_m)=(-1)^{mN}
s_{(N^m)}(\boldsymbol\rho^{-1},\boldsymbol\rho),
\end{equation}
where \((\boldsymbol\rho^{-1},\boldsymbol\rho)\) denotes the concatenated
\(2m\)-tuple, \((N^m)=(N,\dots,N)\) has \(m\) parts, and \(s_\lambda\)
denotes the Schur polynomial indexed by \(\lambda\).  If \(D_N(F_m)\ne0\),
then
\begin{equation}\label{eq:R_N_schur}
R_N=\frac{\displaystyle\prod_{h=1}^m
s_{(N)}(\rho_h^{-1},\rho_h)}
{\displaystyle
s_{(N^m)}(\boldsymbol\rho^{-1},\boldsymbol\rho)}.
\end{equation}
\end{proposition}

\begin{proof}
We give a direct proof using the dual Jacobi--Trudi identity
\cite[Ch.~I, Sec.~3, Eq.~(3.5)]{Macdonald95}.  Put
\(x=(\boldsymbol\rho^{-1},\boldsymbol\rho)\), and define the elementary
symmetric polynomials by
\[
\prod_{j=1}^{2m}(1+x_ju)=\sum_{r=0}^{2m}e_r(x)u^r,
\qquad e_r(x)=0\quad\text{if }r<0\text{ or }r>2m.
\]
Factoring each quadratic gives
\[
F_m(z)=z^{-m}\prod_{j=1}^{2m}(z-x_j)
=\sum_{k=-m}^{m}(-1)^{m-k}e_{m-k}(x)z^k.
\]
Thus the entry in position \((i,j)\) of \(T_N(F_m)\) is
\((-1)^{m-i+j}e_{m-i+j}(x)\).  Factoring the signs from the rows and
columns yields
\[
D_N(F_m)=(-1)^{mN}\det[e_{m-i+j}(x)]_{i,j=1}^{N}.
\]
The partition \((N^m)\) has conjugate partition \((m^N)\).  The dual
Jacobi--Trudi identity therefore identifies the last determinant with
\(s_{(N^m)}(x)\), proving \eqref{eq:rectangular_schur}.
All entries are polynomials in the variables \(x_j\), so this argument
also applies when roots coincide.

To relate this derivation to the Toeplitz minor identity already cited, put
\[
g(z):=\prod_{h=1}^m(1-\rho_hz)(1-\rho_hz^{-1}).
\]
Since \(F_m=(-1)^m(\prod_h\rho_h^{-1})g\), the same formula reads
\[
D_N(g)=\Big(\prod_{h=1}^m\rho_h\Big)^N
s_{(N^m)}(\boldsymbol\rho^{-1},\boldsymbol\rho).
\]
This is precisely the specialization of \cite[Eq.~(19)]{GGT20} with
\(j=k=0\) and both sets of variables equal to
\(-\boldsymbol\rho\), after using homogeneity to remove the signs.
Finally, applying \eqref{eq:rectangular_schur} with \(m=1\) gives
\(D_N(f_h)=(-1)^Ns_{(N)}(\rho_h^{-1},\rho_h)\).
Multiplication over \(h\) and division by \eqref{eq:rectangular_schur}
prove \eqref{eq:R_N_schur} whenever its denominator is nonzero.
\end{proof}

Under the additional assumptions that the \(\rho_h\) are pairwise distinct
and lie in the open unit disk, the next proposition rewrites these expressions
using an \(m\times m\) determinant involving the Cauchy matrix.  Because the
symbol is rational and each quadratic
factor has the factorization \eqref{eq:fh_factor}, the kernel entering the
Borodin--Okounkov--Geronimo--Case Fredholm determinant has finite rank.  The
result specializes classical formulas for rational symbols and root formulas
for banded Toeplitz determinants to this family
\cite{Day75,BasorForrester94,Widom58,SchmidtSpitzer60,BaxterSchmidt61,Trench85,BG05}.
For general symmetric banded Toeplitz matrices, Chebyshev determinant formulas
of fixed order in terms of the reciprocal root pairs of the Laurent symbol are
available in \cite[Thm.~1]{Elouafi14}.  Related Fredholm expansions for finite
rank kernels associated with rational symbols are developed in
\cite[Sec.~6]{BasorConrey24}.

All asymptotic estimates in this subsection keep \(m\) and the roots
\(\rho_1,\dots,\rho_m\) fixed while \(N\to\infty\).
More generally, higher terms in the Szeg\H{o}--Widom asymptotic expansion for
symbols in Wiener algebras with power weights were established in
\cite[Thm.~8]{BottcherSilbermann80}.  Extensions to more general
weights, with remainders controlled by the corresponding weight tails, are
given in \cite[Thm.~1.4(b),(c)]{Karlovich07}.  For the present rational
family, the reduction instead gives the exact formula below for every \(N\)
and identifies the first exponentially small correction.

\begin{proposition}[Exact formula for finite \texorpdfstring{$N$}{N}]\label{prop:finiteN_root}
Let \(N,m\ge1\), and assume that
\(0<|\rho_h|<1\) for \(h=1,\dots,m\) and that the \(\rho_h\) are pairwise
distinct.  Let \(a_h\), \(f_h\), and \(F_m\) be as in
Proposition~\ref{prop:rectangular_schur}, and let
\begin{equation}\label{eq:g_symbol}
g(z):=\prod_{h=1}^m (1-\rho_h z)(1-\rho_h z^{-1}),
\qquad
\Phi(z):=\prod_{h=1}^m(1-\rho_h z).
\end{equation}
Thus \(g\) is palindromic in the sense that \(g(z)=g(z^{-1})\).
Define the weights
\begin{equation}\label{eq:beta_weights}
\beta_h := \rho_h^{m-1}\,\frac{\prod_{j=1}^m (1-\rho_h\rho_j)}{\prod_{j\ne h}(\rho_h-\rho_j)}
= \rho_h^{m-1}(1-\rho_h^2)\prod_{j\ne h}\frac{1-\rho_h\rho_j}{\rho_h-\rho_j},
\qquad h=1,\dots,m,
\end{equation}
the Cauchy matrix
\begin{equation}\label{eq:C_matrix}
\mathcal C_{ij}:=\frac{1}{1-\rho_i\rho_j}\qquad (1\le i,j\le m),
\end{equation}
and the diagonal matrix
\begin{equation}\label{eq:Omega_matrix}
\Omega_N := \operatorname{diag}(\beta_1\rho_1^{N+1},\dots,
\beta_m\rho_m^{N+1}).
\end{equation}
Then the Toeplitz determinant of $g$ admits the exact representation
\begin{equation}\label{eq:detTg_exact}
\det T_N(g) =
\Big(\prod_{i,j=1}^m (1-\rho_i\rho_j)\Big)^{-1}\,
\det\big(I_m - (\Omega_N\mathcal C)^2\big),
\end{equation}
where \((\Omega_N\mathcal C)^2\) denotes ordinary matrix multiplication,
not entrywise squaring.
Put
\begin{equation}\label{eq:delta_alpha_def}
\delta_N:=\det\big(I_m-(\Omega_N\mathcal C)^2\big),
\qquad
\alpha_N:=
\prod_{h=1}^m(1-\rho_h^{2(N+1)})
\prod_{1\le i<j\le m}(1-\rho_i\rho_j)^2.
\end{equation}
Equivalently, \eqref{eq:detTg_exact} and
Proposition~\ref{prop:rectangular_schur} give
\begin{equation}\label{eq:delta_schur}
\delta_N=\Big(\prod_{h=1}^m\rho_h^N\Big)
\Big(\prod_{i,j=1}^m(1-\rho_i\rho_j)\Big)
s_{(N^m)}(\boldsymbol\rho^{-1},\boldsymbol\rho).
\end{equation}
Then the following identity holds without any assumption that \(D_N(F_m)\)
is nonzero:
\begin{equation}\label{eq:finiteN_cross}
\delta_N\,\det\!\Big(\prod_{h=1}^mT_N(f_h)\Big)
=\alpha_N\,D_N(F_m).
\end{equation}
If $D_N(F_m)\ne0$, the ratio
\begin{equation}\label{eq:R_N_def}
R_N:=\frac{\det\big(\prod_{h=1}^m T_N(f_h)\big)}{\det T_N(F_m)},
\end{equation}
is therefore well defined and is given by
\begin{equation}\label{eq:R_N_exact}
R_N
=
\frac{\alpha_N}{\delta_N}.
\end{equation}
If, in addition, \(m\ge2\) and $N\ge m-1$, this ratio agrees with the
$2(m-1)\times2(m-1)$ boundary determinant in
Proposition~\ref{prop:finiteN_reduction}, and hence with the product of the
two $(m-1)\times(m-1)$ determinants in
Corollary~\ref{cor:parity_factorization}.
In particular, if $r=\max_h|\rho_h|<1$, then, along values of $N$ for which
the ratio is defined,
\begin{equation}\label{eq:R_N_exp_small}
R_N = \prod_{1\le i<j\le m}(1-\rho_i\rho_j)^2\;+\;\cO(r^{2N+2}).
\end{equation}
\end{proposition}

\begin{proof}
Using \eqref{eq:fh_factor}, one checks that
\(
F_m(z)=(-1)^m(\prod_h \rho_h^{-1})\,g(z).
\)
Thus
\[
\det T_N(F_m)=(-1)^{mN}(\prod_h\rho_h^{-1})^N\det T_N(g).
\]
Similarly,
\(
\det(\prod_{h=1}^m T_N(f_h))=\prod_{h=1}^m \det T_N(f_h)
\)
and
\[
\det T_N(f_h)=(-\rho_h^{-1})^N
\det T_N\big((1-\rho_h z)(1-\rho_h z^{-1})\big).
\]
The last determinant is the elementary tridiagonal Toeplitz determinant
\begin{equation}\label{eq:scalar_quad_det}
\det T_N\big((1-\rho z)(1-\rho z^{-1})\big)
=\frac{1-\rho^{2(N+1)}}{1-\rho^2}.
\end{equation}
Consequently \eqref{eq:finiteN_cross}, and then
\eqref{eq:R_N_exact} whenever the ratio is defined, follow from
\eqref{eq:detTg_exact}.  It remains to prove \eqref{eq:detTg_exact}.

The finite Toeplitz matrices above use indices $1,\dots,N$, whereas the
infinite Hankel operators in the Borodin--Okounkov formula use indices
$0,1,\dots$.  With
\[
g_-(z):=\Phi(z^{-1}),\qquad g_+(z):=\Phi(z),
\]
one Wiener--Hopf quotient is
\[
b(z):=\frac{g_-(z)}{g_+(z)}
=\frac{\Phi(z^{-1})}{\Phi(z)}
=\sum_{n\in\mathbb Z}b_nz^n .
\]
The other Wiener--Hopf quotient satisfies
\[
\left(\frac{g_+}{g_-}\right)(z^{-1})=b(z).
\]
Thus the two Hankel operators in the Borodin--Okounkov--Geronimo--Case identity
coincide for this palindromic symbol.  Write
\[
H(b):=[b_{i+j+1}]_{i,j\ge0}.
\]
Here the indices start at zero, whereas those of the finite Hankel matrix
\(H_N(G)=[g_{i+j}]_{i,j=1}^N\) used in
\eqref{eq:HN_JN_definition_intro} start at one.
A partial fraction decomposition gives
\[
\frac{1}{\Phi(z)}
=\sum_{h=1}^m
\frac{\rho_h^{m-1}}{\prod_{j\ne h}(\rho_h-\rho_j)}
\frac{1}{1-\rho_hz}.
\]
Multiplying by \(\Phi(z^{-1})=\prod_j(1-\rho_jz^{-1})\) and expanding
the nonnegative Fourier modes gives, for every \(n\ge0\),
\begin{equation}\label{eq:b_coeff_expansion}
b_n=\sum_{h=1}^m
\frac{\rho_h^{m-1}\Phi(\rho_h)}
{\prod_{j\ne h}(\rho_h-\rho_j)}\,\rho_h^n
=\sum_{h=1}^m \beta_h\rho_h^n.
\end{equation}
Hence \(H(b)\) has rank at most \(m\).  Only the rationality of \(b\) is used
for this rank bound.  The same argument using partial fractions shows that the Hankel
operator associated with any Wiener--Hopf quotient of two polynomials of
degree \(m\) has rank at most \(m\).  The palindromic property
\(g(z)=g(z^{-1})\) is used only to identify the second Hankel operator with
the first.  More precisely,
\[
H(b)=V\Lambda V^\top,
\qquad
V_{p,h}=\rho_h^p\ (p\ge0),
\qquad
\Lambda=\operatorname{diag}(\beta_h\rho_h)_{h=1}^m.
\]
Since \(V^\top V=\mathcal C\), this gives
\(H(b)^2=V\Lambda\mathcal C\Lambda V^\top\).

The Borodin--Okounkov--Geronimo--Case identity \cite[Thm.~1]{BO00}, applied
to this scalar palindromic symbol, gives
\begin{equation}\label{eq:detTg_Hankel}
\det T_N(g)=\det(I-H(b)^2)^{-1}\,
\det(I-Q_NH(b)^2Q_N),
\end{equation}
where $Q_N$ is the projection in $\ell^2(\mathbb Z_{\ge0})$ onto the
coordinates $N,N+1,\dots$.
Both operator determinants in \eqref{eq:detTg_Hankel} are Fredholm
determinants, well defined because \(H(b)\) has finite rank.
A short direct proof of the identity using Jacobi's theorem is given in
\cite[Sec.~2, p.~94]{Bottcher02}.

\pagebreak[3]
The factor involving the geometric mean in the general identity equals \(1\), because the
displayed factorization gives
\(
\exp((\log g)_0)=1.
\)
The first factor in \eqref{eq:detTg_Hankel} is the strong Szeg\H{o} constant
for \(g\),
\[
\det(I-H(b)^2)^{-1}
=\prod_{i,j=1}^m(1-\rho_i\rho_j)^{-1},
\]
as computed in the proof of Theorem~\ref{thm:corner_constant_roots}.

After reindexing the range of \(Q_N\) by \(p=0,1,2,\dots\), the matrix \(V\)
satisfies
\[
(Q_NV)_{N+p,h}=\rho_h^{N+p}=\rho_h^N V_{p,h}.
\]
Since \(\rho_h^N\Lambda_{hh}=\beta_h\rho_h^{N+1}\), the compression
\(Q_NH(b)^2Q_N\) is unitarily equivalent to
\[
V\,(\Omega_N\mathcal C\Omega_N)\,V^\top.
\]
Sylvester's determinant identity and \(V^\top V=\mathcal C\) therefore give
\[
\det(I-Q_NH(b)^2Q_N)
=\det(I_m-\Omega_N\mathcal C\Omega_N\mathcal C)
=\det(I_m-(\Omega_N\mathcal C)^2).
\]
Substitution into \eqref{eq:detTg_Hankel} proves \eqref{eq:detTg_exact}.
Since
\(
D_N(g)=(\prod_h\rho_h^N)
s_{(N^m)}(\boldsymbol\rho^{-1},\boldsymbol\rho)
\)
by the proof of Proposition~\ref{prop:rectangular_schur}, comparison with
\eqref{eq:detTg_exact} gives \eqref{eq:delta_schur}.
Finally, \(\|\Omega_N\mathcal C\|=\cO(r^{N+1})\) implies
\(\det(I_m-(\Omega_N\mathcal C)^2)=1+\cO(r^{2N+2})\), giving
\eqref{eq:R_N_exp_small}.
\end{proof}

\begin{remark}[Repeated roots]\label{rem:confluent_roots}
Equations \eqref{eq:rectangular_schur} and \eqref{eq:R_N_schur} do not require
the \(\rho_h\) to be distinct.  The ratio identity applies whenever its
denominator is nonzero.  The weights in \eqref{eq:beta_weights}, however,
generally diverge when roots coalesce, so an explicit confluent matrix
representation requires replacing the geometric sequences by sequences of
the form \(p^a r^p\).  We give that representation here.

Let \(r_1,\dots,r_s\) be the distinct values among the roots, with
multiplicities \(\nu_1,\dots,\nu_s\), where
\(0<|r_\ell|<1\) and \(\sum_{\ell=1}^s\nu_\ell=m\).  Set
\[
\Phi(z)=\prod_{\ell=1}^s(1-r_\ell z)^{\nu_\ell},
\qquad b(z)=\frac{\Phi(z^{-1})}{\Phi(z)},
\qquad g(z)=\Phi(z)\Phi(z^{-1}).
\]
For \(1\le k\le\nu_\ell\), define
\begin{equation}\label{eq:confluent_partial_fractions}
d_{\ell,k}:=
\frac{(-r_\ell^{-1})^{\nu_\ell-k}}{(\nu_\ell-k)!}
\left.
\frac{d^{\nu_\ell-k}}{dz^{\nu_\ell-k}}
\bigl[(1-r_\ell z)^{\nu_\ell}b(z)\bigr]
\right|_{z=r_\ell^{-1}},
\qquad
p_\ell(x):=\sum_{k=1}^{\nu_\ell}
d_{\ell,k}\binom{x+k-1}{k-1}.
\end{equation}
Here \(\binom{x+k-1}{k-1}\) denotes the polynomial
\((x+1)\cdots(x+k-1)/(k-1)!\), with value \(1\) when \(k=1\).
The derivative is well defined because its argument is analytic at
\(r_\ell^{-1}\).  Partial fractions at the poles \(r_\ell^{-1}\)
give the terms \(d_{\ell,k}(1-r_\ell z)^{-k}\).  The remaining terms
are negative powers of \(z\).  Consequently, for every \(n\ge0\),
\begin{equation}\label{eq:confluent_b_coefficients}
b_n=\sum_{\ell=1}^s p_\ell(n)r_\ell^n,
\qquad \deg p_\ell\le\nu_\ell-1.
\end{equation}

Index the next matrices by the \(m\) pairs \((\ell,a)\) with
\(1\le\ell\le s\) and \(0\le a<\nu_\ell\).  Let \(\Lambda\) be
block diagonal, with blocks
\begin{equation}\label{eq:confluent_Lambda}
\Lambda_{(\ell,a),(\ell,c)}
=\frac{r_\ell\,p_\ell^{(a+c)}(1)}{a!\,c!},
\qquad 0\le a,c<\nu_\ell,
\end{equation}
where derivatives above the degree of \(p_\ell\) are zero.
For \(N\ge0\), define
\begin{equation}\label{eq:confluent_Gram}
(\mathcal G_N)_{(\ell,a),(j,c)}
:=\sum_{p=N}^{\infty}p^{a+c}(r_\ell r_j)^p
=\left.
\left(t\frac{d}{dt}\right)^{a+c}\frac{t^N}{1-t}
\right|_{t=r_\ell r_j}.
\end{equation}
The convention is \(p^0=1\), including at \(p=0\).
The convergent sums in \eqref{eq:confluent_Gram} are thus given by explicit
derivatives of a rational function.  The exact determinant formula is
\begin{equation}\label{eq:confluent_determinant}
D_N(g)=
\prod_{\ell,j=1}^s(1-r_\ell r_j)^{-\nu_\ell\nu_j}
\det\!\bigl(I_m-\Lambda\mathcal G_0\Lambda\mathcal G_N\bigr),
\qquad N\ge1.
\end{equation}

To verify it, put \(V_{p,(\ell,a)}=p^ar_\ell^p\) for \(p\ge0\).
Taylor expansion of \(r_\ell p_\ell(1+u+v)\), together with
\eqref{eq:confluent_b_coefficients}, gives
\[
H(b)=V\Lambda V^\top,
\qquad V^\top V=\mathcal G_0,
\qquad V^\top Q_NV=\mathcal G_N.
\]
Applying Sylvester's determinant identity to the Fredholm factor in
\eqref{eq:detTg_Hankel} yields
\[
\det(I-Q_NH(b)^2Q_N)
=\det\!\bigl(I_m-\Lambda\mathcal G_0\Lambda\mathcal G_N\bigr).
\]
The Szeg\H{o} factor is the product in \eqref{eq:confluent_determinant},
by the same Fourier coefficient calculation as before.  This proves the
formula.  When every \(\nu_\ell=1\), one has
\(p_\ell=\beta_\ell\), \(\Lambda_{\ell\ell}=\beta_\ell r_\ell\),
and \((\mathcal G_N)_{\ell j}=(r_\ell r_j)^N/(1-r_\ell r_j)\).
A cyclic permutation of the factors inside the determinant recovers
\eqref{eq:detTg_exact}.

For repeated roots, define
\(\delta_N:=\det(I_m-\Lambda\mathcal G_0\Lambda\mathcal G_N)\).
The identities \eqref{eq:finiteN_cross} and \eqref{eq:R_N_exact} then
hold with this \(\delta_N\) and with \(\alpha_N\) from
\eqref{eq:delta_alpha_def}, counting every root with its multiplicity.
Thus \eqref{eq:confluent_determinant} gives the confluent value of the
determinant, without taking limits of the divergent individual weights.
For fixed roots, each entry of \(\mathcal G_N\) is a polynomial in \(N\)
times \((r_\ell r_j)^N\).  This accounts for the polynomial factors in
the error estimate.  In particular, for any fixed \(r_+\) satisfying
\(\max_\ell|r_\ell|<r_+<1\), the error in the limiting ratio is
\(\cO(r_+^{\,2N})\).
\end{remark}

\begin{remark}[Explicit leading correction for finite $N$]\label{rem:leading_finiteN}
Under the assumptions of Proposition~\ref{prop:finiteN_root}, and for values
of \(N\) for which
\(D_N(F_m)\ne0\), let
$r:=\max_{1\le h\le m}|\rho_h|<1$.
Since \(\|\Omega_N\mathcal C\|=\cO(r^{N+1})\), expansion of the
\(m\times m\) determinant gives
\begin{equation}\label{eq:det_inv_expand}
\det(I_m-(\Omega_N\mathcal C)^2)^{-1}
=1+\operatorname{Tr}((\Omega_N\mathcal C)^2)+\cO(r^{4N+4}).
\end{equation}
Substituting this expansion into \eqref{eq:R_N_exact} gives the explicit
first correction:
\begin{equation}\label{eq:R_N_leading}
R_N
=\prod_{1\le i<j\le m}(1-\rho_i\rho_j)^2\Big(1-\sum_{h=1}^m \rho_h^{2(N+1)}+\operatorname{Tr}((\Omega_N\mathcal C)^2)+\cO(r^{4N+4})\Big).
\end{equation}
The trace can be written in closed form as
\begin{equation}\label{eq:trace_OmegaC2}
\operatorname{Tr}((\Omega_N\mathcal C)^2)=\sum_{i,j=1}^m
\frac{\beta_i\beta_j(\rho_i\rho_j)^{N+1}}
{(1-\rho_i\rho_j)^2}.
\end{equation}
\end{remark}

\paragraph{Numerical verification.}
For the domino parameters \(a_h=2\ii\cos(h\pi/(2m+1))\), we computed the
determinants in \eqref{eq:R_N_def} from their complex phase and logarithmic
modulus and evaluated the expression in \eqref{eq:R_N_exact} using the
\(m\times m\) determinant that defines \(\delta_N\).  Both calculations used
IEEE double precision and agree to the expected accuracy of double precision
arithmetic, as shown in
\Cref{tab:RN}.

\begin{table}[H]
\centering
\begin{tabular}{ccccc}
\hline
$m$ & $N$ & $R_N$ (direct) & $R_N$ (roots) & relative error \\
\hline
2 & 20 & 1.8279586179 & 1.8279586179 & $4.13\times10^{-15}$ \\
3 & 30 & 5.9768427840 & 5.9768427840 & $1.49\times10^{-15}$ \\
4 & 40 & 34.6072635930 & 34.6072635930 & $3.10\times10^{-14}$ \\
\hline
\end{tabular}
\caption{The two evaluations of the ratio \(R_N\), from \eqref{eq:R_N_def} and
from \eqref{eq:R_N_exact}, for the domino parameters
\(a_h=2\ii\cos(h\pi/(2m+1))\), in IEEE double precision.}
\label{tab:RN}
\end{table}

The ancillary script \texttt{verify\_toeplitz\_determinants.py} recomputes the
values in \Cref{tab:RN} in IEEE double precision and the convergence data of
\Cref{fig:convergence} with 120 digits of precision.  It also checks
\eqref{eq:ToeplitzHankelIdentity} for random complex coefficients and verifies
the determinant identity \eqref{eq:dual_cauchy_toeplitz}, obtained from the
dual Cauchy identity, for several rectangular sizes.  The script is supplied as an ancillary file with this
preprint.
\Cref{fig:convergence} plots the relative error of \(R_N\) against its limit
over a range of \(N\), together with the leading correction
\eqref{eq:R_N_leading}.  The determinant data in the figure were recomputed
independently with 120 digits of precision to resolve the exponentially small
tail.

\begin{figure}[!ht]
\centering
\begin{tikzpicture}
\begin{semilogyaxis}[
  width=11.4cm, height=7.2cm,
  xlabel={$N$}, ylabel={$\bigl|R_N/\kappa_m-1\bigr|$},
  xmin=0, xmax=46, ymin=1e-12, ymax=6,
  xtick={0,8,16,24,32,40}, minor xtick={4,12,20,28,36,44},
  ytick={1e0,1e-2,1e-4,1e-6,1e-8,1e-10,1e-12},
  grid=both, grid style={gray!16}, tick label style={font=\small},
  label style={font=\small},
  legend style={font=\small,at={(0.025,0.03)},anchor=south west,draw=gray!50,
                fill=white,cells={anchor=west},row sep=0.5pt},
  legend cell align=left,
]
\addplot[accent,thick,mark=*,mark size=1.5pt] coordinates {
(2,0.5441355936)(4,0.3435552548)(6,0.1507321666)(8,0.05239297003)(10,0.01642205547)
(12,0.004957447926)(14,0.001477522001)(16,0.0004384579183)(18,0.0001299188907)
(20,3.84756872e-5)(22,1.139243216e-5)(24,3.37298976e-6)(26,9.9862285e-7)
(28,2.956536431e-7)(30,8.753124013e-8)(32,2.591446013e-8)(34,7.672218079e-9)
(36,2.271431188e-9)(38,6.724781493e-10)(40,1.990933474e-10)(42,5.894341726e-11)
(44,1.745074083e-11)};
\addlegendentry{$m=2$\ \ ($q_*=0.7376$)}
\addplot[cornercol,thick,mark=square*,mark size=1.4pt] coordinates {
(2,0.939962263)(4,0.9276191994)(6,0.8771601147)(8,0.7721873257)(10,0.5992879746)
(12,0.3893926949)(14,0.2110974095)(16,0.1002558784)(18,0.04419070859)
(20,0.01879102292)(22,0.00786391749)(24,0.00326857576)(26,0.00135464533)
(28,0.0005607473481)(30,0.0002320004319)(32,9.596617073e-5)(34,3.969253792e-5)
(36,1.641660403e-5)(38,6.789705496e-6)(40,2.808120008e-6)(42,1.161392906e-6)
(44,4.803327024e-7)};
\addlegendentry{$m=3$\ \ ($q_*=0.8019$)}
\addplot[black!72,thick,mark=triangle*,mark size=1.9pt] coordinates {
(4,0.9949306144)(6,0.9936316945)(8,0.9901429574)(10,0.9829564268)(12,0.9688562686)
(14,0.9419303847)(16,0.892514424)(18,0.8080116419)(20,0.6796654195)(22,0.5161044006)
(24,0.3486661203)(26,0.2116364762)(28,0.1186037958)(30,0.06317495159)
(32,0.03268564315)(34,0.01664834839)(36,0.008411060085)(38,0.004231802608)
(40,0.002124647462)(42,0.001065585961)(44,0.0005341448738)};
\addlegendentry{$m=4$\ \ ($q_*=0.8413$)}
\addplot[white,line width=1.6pt,forget plot] coordinates {
(8,0.05552098841)(10,0.01671690696)(12,0.004983972163)(14,0.001479868818)
(16,0.0004386643356)(18,0.0001299370072)(20,3.847727595e-5)(22,1.139257145e-5)
(24,3.373001969e-6)(26,9.986239202e-7)(28,2.956537369e-7)(30,8.753124836e-8)
(32,2.591446085e-8)(34,7.672218142e-9)(36,2.271431193e-9)(38,6.724781498e-10)
(40,1.990933474e-10)(42,5.894341726e-11)(44,1.745074084e-11)};
\addplot[black,dash pattern=on 2.6pt off 2.2pt,line width=.6pt] coordinates {
(8,0.05552098841)(10,0.01671690696)(12,0.004983972163)(14,0.001479868818)
(16,0.0004386643356)(18,0.0001299370072)(20,3.847727595e-5)(22,1.139257145e-5)
(24,3.373001969e-6)(26,9.986239202e-7)(28,2.956537369e-7)(30,8.753124836e-8)
(32,2.591446085e-8)(34,7.672218142e-9)(36,2.271431193e-9)(38,6.724781498e-10)
(40,1.990933474e-10)(42,5.894341726e-11)(44,1.745074084e-11)};
\addlegendentry{leading correction \eqref{eq:R_N_leading}}
\addplot[white,line width=1.6pt,forget plot] coordinates {
(16,0.1114884444)(18,0.04624434057)(20,0.01915269111)(22,0.007926557557)
(24,0.003279347204)(26,0.00135649192)(28,0.000561063508)(30,0.000232054533)
(32,9.597542635e-5)(34,3.969412122e-5)(36,1.641687487e-5)(38,6.789751823e-6)
(40,2.808127933e-6)(42,1.161394262e-6)(44,4.803329342e-7)};
\addplot[black,dash pattern=on 2.6pt off 2.2pt,line width=.6pt,forget plot] coordinates {
(16,0.1114884444)(18,0.04624434057)(20,0.01915269111)(22,0.007926557557)
(24,0.003279347204)(26,0.00135649192)(28,0.000561063508)(30,0.000232054533)
(32,9.597542635e-5)(34,3.969412122e-5)(36,1.641687487e-5)(38,6.789751823e-6)
(40,2.808127933e-6)(42,1.161394262e-6)(44,4.803329342e-7)};
\addplot[white,line width=1.6pt,forget plot] coordinates {
(26,0.2684741711)(28,0.1345695211)(30,0.06743666585)(32,0.03379047091)
(34,0.01693030298)(36,0.008482429855)(38,0.004249792826)(40,0.002129172696)
(42,0.001066723021)(44,0.0005344304312)};
\addplot[black,dash pattern=on 2.6pt off 2.2pt,line width=.6pt,forget plot] coordinates {
(26,0.2684741711)(28,0.1345695211)(30,0.06743666585)(32,0.03379047091)
(34,0.01693030298)(36,0.008482429855)(38,0.004249792826)(40,0.002129172696)
(42,0.001066723021)(44,0.0005344304312)};
\end{semilogyaxis}
\end{tikzpicture}
\caption{Approach of the ratio $R_N$ of \eqref{eq:R_N_def} to its limit
$\kappa_m=\prod_{i<j}(1-\rho_i\rho_j)^2$, for the domino parameters
$a_h=2\ii\cos(h\pi/(2m+1))$ and even $N$.  Markers: the relative error
computed from the determinants themselves.  Dashed: the leading correction
\eqref{eq:R_N_leading}, plotted where it is below $1$.  The tails are asymptotically straight:
Corollary~\ref{cor:domino_leading_sign} gives slope $2\log_{10}q_*$ for
$\log_{10}|R_N/\kappa_m-1|$ as a function of $N$.
Here $q_*=\max_hq_h=|\rho_m|$.  Both this rate and the coefficient $A_m$
in \eqref{eq:domino_leading_coefficient} determine the onset of the
asymptotic regime.  All computations
were carried out with 120 digits of precision.}
\label{fig:convergence}
\end{figure}
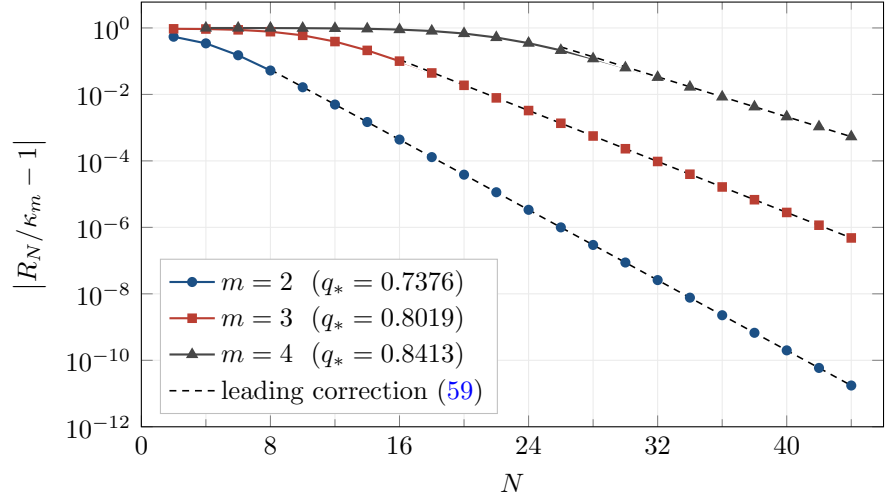

\section{Rectangular domino tilings}\label{sec:domino}

Rectangular domino tilings provide a concrete application.  The enumeration
used below is classical.  The contribution of this section is the exact
comparison between the product of the \(m\) finite tridiagonal Toeplitz
matrices arising from the Kasteleyn diagonalization and the single banded
Toeplitz matrix formed from their product symbol.
Theorem~\ref{thm:domino_root_correction} expresses the correction through an
\(m\times m\) determinant for every \(m\) and \(n\), and gives its first
exponentially small term for fixed \(m\).  The dual Cauchy identity gives
a Toeplitz determinant equal to the tiling number.  The banded
determinant enters through the stated correction factor.  The identity with Hankel
corrections in Section~\ref{sec:corner} applies directly to the same
comparison in the range \(m\le2n+2\), as recorded in
Remark~\ref{rem:domino_band}.
For background on dimer models and on the
enumeration of perfect matchings, see \cite{Kenyon09,Propp99,Strehl01} and, for the
plane lattice specifically, \cite{FisherStephenson63}.

We use the following form of Gessel's identity in terms of a Toeplitz
determinant, written using the generating functions for complete homogeneous and elementary
symmetric polynomials.  Here an \emph{alphabet} means a finite list of
variables.  For \(x=(x_1,\dots,x_p)\) and \(y=(y_1,\dots,y_q)\), write
\[
\mathcal H(y;z)=\prod_{j=1}^q(1-y_jz)^{-1}
=\sum_{r\ge0}h_r(y)z^r,
\qquad
E(x;z)=\prod_{i=1}^p(1+x_iz)
=\sum_{r\ge0}e_r(x)z^r.
\]
Thus \(h_r\) and \(e_r\) are, respectively, the complete homogeneous and
elementary symmetric polynomials.  The definitions
\(T_N(G)=[g_{i-j}]_{i,j=1}^N\) and \(D_N(G)=\det T_N(G)\) extend verbatim to
a formal Laurent series \(G(z)=\sum_kg_kz^k\) whenever its coefficients are
well defined.  In this sense, Gessel's formula gives
\begin{equation}\label{eq:gessel}
D_N\!\big(\mathcal H(y;z^{-1})E(x;z)\big)=
\sum_{\lambda:\ \lambda_1\le N}s_\lambda(x)s_{\lambda'}(y),
\end{equation}
where a partition \(\lambda=(\lambda_1\ge\lambda_2\ge\cdots)\) is a finite
weakly decreasing sequence of nonnegative integers, \(\lambda'\) is the
partition obtained by transposing its Young diagram, and \(s_\lambda\) is the
Schur polynomial.  In particular, \(s_\lambda(x)=0\) if \(\lambda\) has more
than \(p\) nonzero parts.  We also use the dual Cauchy identity
\begin{equation}\label{eq:dual_cauchy}
\sum_\lambda s_\lambda(x)s_{\lambda'}(y)
=\prod_{i=1}^p\prod_{j=1}^q(1+x_iy_j).
\end{equation}
See \cite[Ch.~I]{Macdonald95} and \cite{Gessel90}.  For the asymptotic use of
\eqref{eq:gessel}, see \cite{BDJ99}.  Toeplitz minors correspond
more generally to skew Schur generating functions \cite{BD02}.  Related
minors of matrices of the form Toeplitz \(+\) Hankel or Toeplitz \(-\) Hankel
occur in matrix models for classical groups \cite{GGT20,GGTTH20}.

Let \(m,n\ge1\).  As in the introduction, let \(Z_{m,n}\) denote the number
of tilings of a \(2m\times2n\) board of unit squares by \(2\times1\)
dominoes.  Set
\[
\xi_{h,m}:=\cos\frac{h\pi}{2m+1},\qquad
\eta_{k,n}:=\cos\frac{k\pi}{2n+1},
\]
and define
\(
\mathsf A_h:=4\xi_{h,m}^2
\)
and
\(
\mathsf B_k:=4\eta_{k,n}^2.
\)
The Kasteleyn--Temperley--Fisher product \cite{Kas61,TF61,Fisher61} is
\begin{equation}\label{eq:Zmn}
Z_{m,n}=\prod_{h=1}^{m}\prod_{k=1}^{n}\big(\mathsf A_h+\mathsf B_k\big).
\end{equation}

\paragraph{Toeplitz determinant from the dual Cauchy identity.}
Put
\[
x=(\mathsf A_1,\dots,\mathsf A_m),\qquad
y=(\mathsf B_1^{-1},\dots,\mathsf B_n^{-1}),\qquad
\mathcal B_n:=\prod_{k=1}^n\mathsf B_k=1.
\]
To see the last equality, put \(\theta_k=k\pi/(2n+1)\) and use
\(2\cos\theta_k=\sin(2\theta_k)/\sin\theta_k\).
Reflecting those angles \(2\theta_k\) that exceed \(\pi/2\) about
\(\pi/2\) gives a permutation of the angles \(\theta_k\), \(1\le k\le n\).
Thus \(\prod_{k=1}^n2\cos\theta_k=1\), and squaring gives
\(\mathcal B_n=1\).
Apply \eqref{eq:gessel} with the two alphabets interchanged and \(N=m\).
With \(\mu=\lambda'\), the expression on the right becomes
\[
\sum_{\mu:\,\ell(\mu)\le m}s_\mu(x)s_{\mu'}(y),
\]
where \(\ell(\mu)\) is the number of nonzero parts of \(\mu\).  Because \(x\)
has \(m\) variables, the same restriction is already imposed by
\(s_\mu(x)=0\).  Equation~\eqref{eq:dual_cauchy} therefore gives
\begin{equation}\label{eq:dual_cauchy_toeplitz}
Z_{m,n}
=D_m\!\Big(\mathcal B_n\,\mathcal H(x;z^{-1})\,E(y;z)\Big).
\end{equation}
Indeed, before multiplication by \(\mathcal B_n\), the determinant equals
\(\prod_{h,k}(1+\mathsf A_h/\mathsf B_k)\).  Multiplying its symbol by
\(\mathcal B_n\) multiplies an \(m\times m\) determinant by
\(\mathcal B_n^m\), yielding \eqref{eq:Zmn}.
Equation~\eqref{eq:dual_cauchy_toeplitz} is valid formally at the level of
Laurent coefficients.  For an analytic interpretation on \(\bbT\), replace
\(x_h\) by \(t^{-1}x_h\) and \(y_k\) by \(t y_k\), with
\(t>\max_h|x_h|\).  The resulting Laurent coefficients satisfy
\(\widetilde g_k=t^kg_k\), so the two Toeplitz matrices are diagonally
similar and have the same determinant.  The rescaling also makes the
geometric series in \(\mathcal H(t^{-1}x;z^{-1})\) converge on \(\bbT\).

\paragraph{Chebyshev representation and tridiagonal Toeplitz product.}
Let \(U_r\) denote the Chebyshev polynomial of the second kind, normalized
by
\[
U_r(\cos\theta)=\frac{\sin((r+1)\theta)}{\sin\theta}.
\]
For every \(\xi\),
\[
\det T_{2n}(z+z^{-1}+2\ii\xi)=U_{2n}(\ii\xi)
=(-1)^n\prod_{k=1}^n(4\xi^2+\mathsf B_k),
\]
where the last equality pairs the roots of $U_{2n}$ at
$\cos(k\pi/(2n+1))$ and $-\cos(k\pi/(2n+1))$.
Therefore, with $a_h:=2\ii\xi_{h,m}$,
\begin{equation}\label{eq:Z_cheb_prod}
Z_{m,n}=(-1)^{mn}\det\!\Big(\prod_{h=1}^{m}T_{2n}\!\big(z+z^{-1}+a_h\big)\Big).
\end{equation}
The sign cancels in the normalized ratio
\((-1)^{mn}Z_{m,n}/D_{2n}(F_m)\), but is required in the exact identity when
\(mn\) is odd.

\begin{theorem}[Exact correction for rectangular domino tilings]
\label{thm:domino_root_correction}
Let \(m,n\ge1\).  For \(h=1,\dots,m\), abbreviate
\begin{equation}\label{eq:domino_q}
\xi_h:=\cos\frac{h\pi}{2m+1},
\qquad
q_h:=\sqrt{1+\xi_h^2}-\xi_h\in(0,1).
\end{equation}
The numbers \(q_1,\dots,q_m\) are pairwise distinct.  Define
\[
F_m(z):=\prod_{h=1}^m(z+z^{-1}+2\ii\xi_h),
\]
and put
\begin{align}
\mathcal C^{\mathrm{dom}}_{hj}&:=\frac{1}{1+q_hq_j},\label{eq:C_dom}\\
\gamma_h&:=
q_h^{m-1}(1+q_h^2)
\prod_{\substack{j=1\\j\ne h}}^m
\frac{1+q_hq_j}{q_h-q_j}
=q_h^{m-1}
\frac{\prod_{j=1}^m(1+q_hq_j)}
{\prod_{j\ne h}(q_h-q_j)},\label{eq:gamma_dom}\\
\mathcal E_{m,n}&:=
\operatorname{diag}\big(
\gamma_1q_1^{2n+1},\dots,\gamma_mq_m^{2n+1}\big),\label{eq:E_dom}\\
\kappa_m&:=\prod_{1\le h<j\le m}(1+q_hq_j)^2.\label{eq:kappa_dom}
\end{align}
Then the following identity holds without any nonvanishing assumption on
\(D_{2n}(F_m)\):
\begin{equation}\label{eq:domino_root_cross}
\det\!\big(I_m+(\mathcal E_{m,n}\mathcal C^{\mathrm{dom}})^2\big)
\,(-1)^{mn}Z_{m,n}
=
\left(\prod_{h=1}^m(1+q_h^{4n+2})\right)
\kappa_m\,D_{2n}(F_m).
\end{equation}
As in \eqref{eq:detTg_exact}, the square denotes ordinary matrix
multiplication, not entrywise squaring.
Whenever $D_{2n}(F_m)\ne0$,
\begin{equation}\label{eq:domino_root_exact}
\frac{(-1)^{mn}Z_{m,n}}{D_{2n}(F_m)}
=
\frac{
\displaystyle
\left(\prod_{h=1}^m(1+q_h^{4n+2})\right)\kappa_m
}{
\det\!\big(I_m+(\mathcal E_{m,n}\mathcal C^{\mathrm{dom}})^2\big)
}.
\end{equation}
The quotient in \eqref{eq:domino_root_exact} is defined for all large \(n\).
Indeed, \(z+z^{-1}\) is real on \(\bbT\), whereas \(2\ii\xi_h\) is purely imaginary
and nonzero because \(0<h\pi/(2m+1)<\pi/2\).  Hence \(f_h\) does not vanish on
\(\bbT\), and Theorem~\ref{thm:corner_constant_roots} gives
\(D_{2n}(F_m)\ne0\) for all sufficiently large \(n\).
Consequently, for fixed $m$ and $n\to\infty$,
\begin{equation}\label{eq:domino_root_limit}
\frac{(-1)^{mn}Z_{m,n}}{D_{2n}(F_m)}
=\kappa_m+\cO(q_*^{4n+2}),
\qquad
q_*:=\max_h q_h
=e^{-\operatorname{arsinh}(\sin(\pi/(4m+2)))}<1 .
\end{equation}
Here \(\operatorname{arsinh}\) denotes the inverse hyperbolic sine.
More precisely,
\begin{align}
\frac{(-1)^{mn}Z_{m,n}}{D_{2n}(F_m)}
=\kappa_m\Bigg[
1+\sum_{h=1}^m q_h^{4n+2}
-\sum_{h,j=1}^m
\frac{\gamma_h\gamma_j(q_hq_j)^{2n+1}}
{(1+q_hq_j)^2}
+\cO(q_*^{8n+4})
\Bigg].
\label{eq:domino_root_first}
\end{align}
All asymptotic estimates are for fixed $m$.
\end{theorem}

\begin{proof}
The identity
\[
\sqrt{1+x^2}-x=e^{-\operatorname{arsinh}x}
\]
shows that \(q_h=e^{-\operatorname{arsinh}\xi_h}\).  The sequence
\(\xi_1,\dots,\xi_m\) is strictly decreasing and positive, so
\(q_1,\dots,q_m\) is strictly increasing and hence pairwise distinct.
Moreover,
\[
\xi_m=\sin\frac{\pi}{4m+2},
\]
which proves the stated formula for \(q_* = q_m\).

For $a_h=2\ii\xi_h$, the root of $z^2+a_hz+1$ in the open unit disk is
\[
\rho_h=\ii q_h.
\]
Substitution into \eqref{eq:beta_weights} gives \(\beta_h=\gamma_h\), while
\[
\mathcal C=\mathcal C^{\mathrm{dom}},
\qquad
\Omega_{2n}=\ii(-1)^n\mathcal E_{m,n}.
\]
Therefore
\[
\det\!\big(I_m-(\Omega_{2n}\mathcal C)^2\big)
=
\det\!\big(I_m+(\mathcal E_{m,n}\mathcal C^{\mathrm{dom}})^2\big),
\]
and
\[
1-\rho_h^{4n+2}=1+q_h^{4n+2},
\qquad
1-\rho_h\rho_j=1+q_hq_j.
\]
Combining Proposition~\ref{prop:finiteN_root} with
\eqref{eq:Z_cheb_prod} proves \eqref{eq:domino_root_cross} and
\eqref{eq:domino_root_exact}.

Finally,
\[
\prod_h(1+q_h^{4n+2})
=1+\sum_h q_h^{4n+2}+\cO(q_*^{8n+4}),
\]
and
\[
\det\!\big(I_m+(\mathcal E_{m,n}\mathcal C^{\mathrm{dom}})^2\big)^{-1}
=1-\operatorname{Tr}\big((\mathcal E_{m,n}\mathcal C^{\mathrm{dom}})^2\big)
+\cO(q_*^{8n+4}).
\]
Expanding the trace gives \eqref{eq:domino_root_first}.
\end{proof}

\begin{corollary}[Leading term for fixed width]\label{cor:domino_leading_sign}
Under the notation of Theorem~\ref{thm:domino_root_correction}, let
\(m\ge2\) be fixed and define
\begin{equation}\label{eq:domino_leading_coefficient}
A_m:=\frac{\gamma_m^2}{(1+q_m^2)^2}-1
=\prod_{j=1}^{m-1}
\left(\frac{q_m(1+q_mq_j)}{q_m-q_j}\right)^2-1.
\end{equation}
Then \(A_m>0\), and, as \(n\to\infty\),
\begin{equation}\label{eq:domino_leading_signed}
\frac{(-1)^{mn}Z_{m,n}}{\kappa_mD_{2n}(F_m)}
=1-A_mq_m^{4n+2}+o(q_m^{4n+2}).
\end{equation}
In particular, the normalized ratio approaches \(1\) from below for all
sufficiently large \(n\).
\end{corollary}

\begin{proof}
Because \(0<q_1<\cdots<q_m<1\), the unique term of largest exponential
order in the first sum of \eqref{eq:domino_root_first} is
\(q_m^{4n+2}\).  In the double sum, it is the term with \(h=j=m\),
which equals \(\gamma_m^2q_m^{4n+2}/(1+q_m^2)^2\).
All other terms, and the remainder \(\cO(q_m^{8n+4})\), are
\(o(q_m^{4n+2})\).  Their difference gives
\eqref{eq:domino_leading_signed}.
The product expression in \eqref{eq:domino_leading_coefficient} follows
from \eqref{eq:gamma_dom}.  For every \(j<m\),
\[
\frac{q_m(1+q_mq_j)}{q_m-q_j}
=1+\frac{q_j(1+q_m^2)}{q_m-q_j}>1,
\]
so \(A_m>0\).  The exact expression \eqref{eq:domino_root_exact}
is real whenever it is defined, since \(\mathcal E_{m,n}\),
\(\mathcal C^{\mathrm{dom}}\), and its scalar factors are real.
It is defined for all sufficiently large \(n\) by
Theorem~\ref{thm:domino_root_correction}, and its leading correction
in \eqref{eq:domino_leading_signed} is strictly negative.
\end{proof}

For widths \(4\), \(8\), and \(16\) (\(m=2,4,8\), respectively), the coefficients are
\[
A_2\approx13.67,\qquad A_4\approx3030.66,\qquad
A_8\approx2.16051\times10^8.
\]
The size of the leading relative correction is \(A_mq_m^{4n+2}\).
Both the coefficient \(A_m\) and the rate \(q_m\) are therefore relevant
when assessing the onset of the asymptotic regime.

\paragraph{Concrete widths.}
The same parameters give a simple exact expression for the tiling number.
Recurrences and rational generating functions for tilings of rectangles of
fixed width are classical \cite{KlarnerPollack80,Stanley85}.
Equations~\eqref{eq:scalar_quad_det} and \eqref{eq:Z_cheb_prod} yield
\begin{equation}\label{eq:Zmn_q_product}
Z_{m,n}=C_m\Lambda_m^n
\prod_{h=1}^m\left(1+q_h^2(q_h^4)^n\right),
\qquad
C_m:=\prod_{h=1}^m\frac{1}{1+q_h^2},
\qquad
\Lambda_m:=\prod_{h=1}^m q_h^{-2}.
\end{equation}
The constants in \eqref{eq:Zmn_q_product} satisfy
\[
\prod_{h=1}^m\bigl(4+4\xi_h^2\bigr)=P_{2m+1},
\qquad
C_m=\sqrt{\frac{\Lambda_m}{P_{2m+1}}},
\]
where \(P_0=0\), \(P_1=1\), \(P_{j+1}=2P_j+P_{j-1}\) are the Pell numbers
\cite{Horadam71}.
Indeed, the pairing of Chebyshev roots used for \eqref{eq:Z_cheb_prod} gives
\(\prod_{h=1}^m(x^2+4\xi_h^2)=(-1)^mU_{2m}(\ii x/2)\).  At \(x=2\), the
sequence \(\ii^{-j}U_j(\ii)\) obeys the Pell recurrence with initial values
\(1\) and \(2\), so \((-1)^mU_{2m}(\ii)=P_{2m+1}\).  The second identity
then follows from \(1+q_h^2=2q_h\sqrt{1+\xi_h^2}\).

Since \(q_1<\cdots<q_m\), for every fixed \(m\ge2\),
\begin{equation}\label{eq:Zmn_fixed_width}
Z_{m,n}=C_m\Lambda_m^n
\left[1+q_m^2(q_m^4)^n
+\cO\big((q_{m-1}^4)^n\big)\right].
\end{equation}
Writing the constants numerically, for \(2m=4,8,16\) these asymptotic
formulas read
\begin{equation}\label{eq:Zmn_concrete}
\begin{aligned}
Z_{2,n}
&=0.5274743291\ldots\,(8.068645869\ldots)^n\\
&\quad{}\times\left[1+0.5441132198\ldots\,(0.2960591960\ldots)^n
 +\cO\big((0.05188240921\ldots)^n\big)\right],\\[2pt]
Z_{4,n}
&=0.2869899903\ldots\,(81.12780576\ldots)^n\\
&\quad{}\times\left[1+0.7078137786\ldots\,(0.5010003452\ldots)^n
 +\cO\big((0.1458980337\ldots)^n\big)\right],\\[2pt]
Z_{8,n}
&=0.08644464036\ldots\,(8494.108436\ldots)^n\\
&\quad{}\times\left[1+0.8317063241\ldots\,(0.6917354095\ldots)^n
 +\cO\big((0.3391104026\ldots)^n\big)\right].
\end{aligned}
\end{equation}
\paragraph{Accuracy of the enumeration formulas.}
The exact product \eqref{eq:Zmn_q_product} applies to every \(m,n\ge1\).
For \(m=4\), it gives the counts for both the \(8\times8\) chessboard and
the \(8\times40\) rectangle:
\begin{equation}\label{eq:chessboard_count}
\begin{aligned}
Z_{4,4}&=12{,}988{,}816,\\
Z_{4,20}&=43\,778\,578\,308\,070\,520\,454\,643\,930\,160\,388\,065\,413.
\end{aligned}
\end{equation}
Both integers were independently checked by a transfer computation using
integer arithmetic.  To assess \eqref{eq:Zmn_concrete}, define
\[
\widehat Z_{4,n}:=C_4\Lambda_4^n(1+q_4^{4n+2}),
\]
which retains the first correction and omits the remainder.
The exact product gives its relative error directly:
\[
\frac{Z_{4,n}-\widehat Z_{4,n}}{Z_{4,n}}
=1-\prod_{h=1}^3(1+q_h^{4n+2})^{-1}>0.
\]
Table~\ref{tab:tiling_accuracy} compares the leading term and this corrected
approximation.  The relative error of \(\widehat Z_{4,n}\) is about
\(0.0176\%\) at \(n=4\) and \(7.29\times10^{-18}\) at \(n=20\).
The constants were evaluated from their exact definitions in
\eqref{eq:domino_q} and \eqref{eq:Zmn_q_product}. The abbreviated decimals
in \eqref{eq:Zmn_concrete} do not retain enough digits to reproduce the
latter accuracy.
The ancillary script \texttt{verify\_toeplitz\_tiling\_comparison.py}
reproduces Table~\ref{tab:tiling_accuracy} and independently checks both
integer counts by transfer enumeration.

\begin{table}[!ht]
\centering
\small
\setlength{\tabcolsep}{4pt}
\renewcommand{\arraystretch}{1.2}
\begin{tabular}{@{}rlll@{}}
\hline
\rule{0pt}{3ex}\(n\) & \(C_4\Lambda_4^n\) & \(\widehat Z_{4,n}\) & Relative error\\
\hline
4 & \(12\,432\,133.628144\)
  & \(12\,986\,525.408831\) & \(1.76351037\times10^{-4}\)\\
20 & \(4.3778547551290596584\times10^{37}\)
   & \(4.3778578308070520135\times10^{37}\)
   & \(7.29480192\times10^{-18}\)\\
\hline
\end{tabular}
\caption{Approximation of the tiling counts in \eqref{eq:chessboard_count}.
The final column is \((Z_{4,n}-\widehat Z_{4,n})/Z_{4,n}\), computed before
rounding the displayed values.}
\label{tab:tiling_accuracy}
\end{table}

\paragraph{Comparison with the banded Toeplitz determinant.}
The formulas for \(Z_{m,n}\) itself are consequences of the classical product.
Theorem~\ref{thm:domino_root_correction} gives the exact comparison with
\(D_{2n}(F_m)\), also for every \(m,n\ge1\) in the form
\eqref{eq:domino_root_cross}.  Combining \eqref{eq:Zmn_q_product} with that
identity shows precisely what its correction determinant measures:
\[
\frac{(-1)^{mn}\kappa_mD_{2n}(F_m)}{C_m\Lambda_m^n}
=\det\!\bigl(I_m+(\mathcal E_{m,n}\mathcal C^{\mathrm{dom}})^2\bigr).
\]
Thus this determinant is the multiplicative correction to the leading term
for \(D_{2n}(F_m)\).  The scalar product in \eqref{eq:Zmn_q_product} plays
the corresponding role for the tiling count.

\begin{samepage}
For \(m=2,4,8\), this gives
\begin{equation}\label{eq:kappa_concrete}
\begin{gathered}
D_{2n}(F_m)\sim\frac{Z_{m,n}}{\kappa_m},\\
\kappa_2=1.828028952\ldots,
\quad
\kappa_4=34.68094838\ldots,
\quad
\kappa_8=11{,}748{,}559.44\ldots.
\end{gathered}
\end{equation}
\end{samepage}
For \(m=n=4\), the banded determinant is
\(D_8(F_4)=37{,}995{,}482.596169\ldots\).
The scalar product and correction determinant are
\[
\begin{aligned}
\prod_{h=1}^4(1+q_h^{18})&=1.044777701761\ldots,\\
\det\!\bigl(I_4+(\mathcal E_{4,4}\mathcal C^{\mathrm{dom}})^2\bigr)
&=105.9930185823\ldots.
\end{aligned}
\]
Equation~\eqref{eq:domino_root_exact} therefore gives
\begin{equation}\label{eq:chessboard_identity}
Z_{4,4}
=D_8(F_4)\,
\frac{\prod_{h=1}^4\bigl(1+q_h^{18}\bigr)\,\kappa_4}
{\det\bigl(I_4+(\mathcal E_{4,4}\mathcal C^{\mathrm{dom}})^2\bigr)}
=12{,}988{,}816=3604^2 .
\end{equation}
The same evaluation at \(n=20\), using \(D_{40}(F_4)\), recovers the second
integer in \eqref{eq:chessboard_count}.
At \(n=4\), replacing the exact ratio \(Z_{4,n}/D_{2n}(F_4)\) by its limit
\(\kappa_4\) overestimates the tiling count by a factor of
\(101.4503\ldots\).  At \(n=20\), the correction determinant is
\(1.002129875\ldots\), and the same replacement has relative error
\(2.12917\times10^{-3}\).  This slower convergence of the determinant ratio,
shown in \Cref{fig:convergence}, is compatible with the much smaller errors
of the enumeration formula in Table~\ref{tab:tiling_accuracy}.

\begin{remark}\label{rem:domino_band}
When \(m\le2n+2\), the matrix identity
\eqref{eq:ToeplitzHankelIdentity} also applies directly to
\eqref{eq:Z_cheb_prod}.  If \(2n\ge m-1\), the resulting reduced determinant
factors further according to
Corollary~\ref{cor:parity_factorization}.
\end{remark}

\Cref{fig:map} summarizes the three evaluations of \(Z_{m,n}\) and their
comparison with the banded Toeplitz determinant \(D_{2n}(F_m)\).

\begin{figure}[!ht]
\centering
\begin{tikzpicture}[
  box/.style={draw=gray!55,rounded corners=2pt,fill=gray!4,align=center,
              inner sep=4pt,font=\small,minimum height=12mm,text width=42mm},
  newbox/.style={box,draw=cornercol,fill=cornerfill!40},
  ar/.style={{Stealth[length=2mm]}-{Stealth[length=2mm]},gray!55!black,line width=.7pt},
  newar/.style={{Stealth[length=2mm]}-{Stealth[length=2mm]},cornercol,line width=1.1pt},
  lab/.style={font=\scriptsize,align=center,inner sep=1.5pt,fill=white},
]
\node[box] (Z) at (0,0)
   {$\displaystyle Z_{m,n}=\prod_{h=1}^m\prod_{k=1}^n(\mathsf A_h+\mathsf B_k)$\\[1pt]
    {\scriptsize Kasteleyn--Temperley--Fisher}};
\node[box] (DC) at (9.4,0)
   {$\displaystyle D_m\!\bigl(\mathcal B_n\mathcal H(x;z^{-1})E(y;z)\bigr)$\\[1pt]
    {\scriptsize $m\times m$ Toeplitz determinant}};
\node[box] (CH) at (0,-3.0)
   {$\displaystyle (-1)^{mn}\det\prod_{h=1}^mT_{2n}(f_h)$\\[1pt]
    {\scriptsize product of tridiagonal sections}};
\node[newbox] (BT) at (9.4,-3.0)
   {$\displaystyle D_{2n}(F_m)$\\[1pt]
    {\scriptsize $2n\times2n$ banded Toeplitz determinant}};
\draw[ar] (Z) -- node[lab,above,yshift=-1pt]{Gessel $+$ dual Cauchy}
                 node[lab,below,yshift=1pt]{\eqref{eq:dual_cauchy_toeplitz}} (DC);
\draw[ar] (Z) -- node[lab,right,xshift=-1pt]
                 {Chebyshev $U_{2n}$\\[-1pt]\eqref{eq:Z_cheb_prod}} (CH);
\draw[newar] (CH) -- node[lab,above,yshift=-1pt,text=cornercol]
                       {\bfseries explicit correction factor}
                     node[lab,below,yshift=1pt,text=cornercol]
                       {\Cref{thm:domino_root_correction}} (BT);
\end{tikzpicture}
\caption{Three evaluations of the number of domino tilings of a
$2m\times2n$ rectangle and the banded Toeplitz determinant compared with them.
Both upper boxes and the lower left box equal $Z_{m,n}$.  The lower right
box is $D_{2n}(F_m)$.  Multiplying it by $(-1)^{mn}$ and the factor on the
right of \eqref{eq:domino_root_exact} gives $Z_{m,n}$.
The top and left links are classical.  The bottom link is established
by combining the comparison of finite sections in
Sections~\ref{sec:corner}--\ref{sec:determinants} with the root formula.  It is
an identity, not an asymptotic statement, and
\eqref{eq:domino_root_exact} expresses its correction factor as an explicit
scalar divided by the determinant of an $m\times m$ matrix.  For fixed $m$,
this matrix dimension is independent of $n$.  Composing the three links
expresses the small Toeplitz determinant $D_m$ in terms of the large one
$D_{2n}$.}
\label{fig:map}
\end{figure}
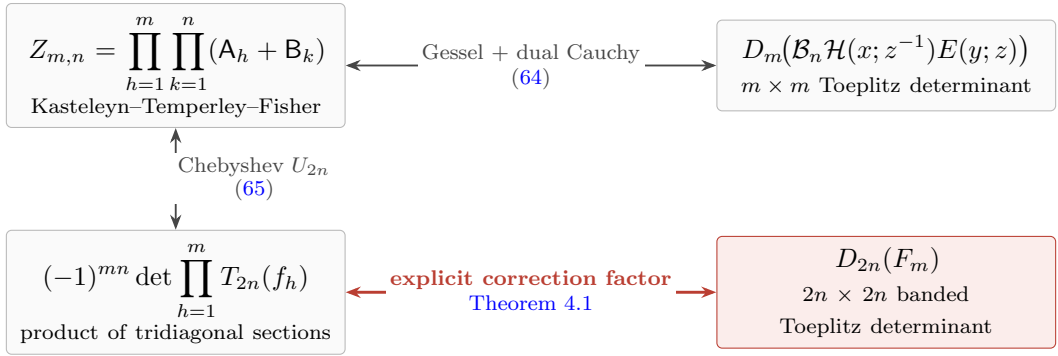

\section{Conclusion and outlook}\label{sec:outlook}
The classical product formula for finite sections gives explicit boundary
corrections for the tridiagonal family symmetric under inversion.  The method
of images and the iteration of that formula establish the sharp uniform
range \(d\le N+2\).  For \(N\ge d-1\), the determinant ratio, whenever
defined, reduces to a determinant of order \(2(d-1)\), which reversal
symmetry splits into two determinants of order \(d-1\).
For \(P(x)=\prod_h(x+a_h)\), the Schur formulas in
Proposition~\ref{prop:rectangular_schur} also cover repeated roots.
When the selected roots \(\rho_h\) lie in the open unit disk, the limiting ratio is
\eqref{eq:corner_constant_product}.  For pairwise distinct roots, classical
identities for rational symbols give the exact formula
\eqref{eq:finiteN_cross} through an \(m\times m\) determinant and the
leading correction \eqref{eq:R_N_leading}.  The explicit confluent matrix
in Remark~\ref{rem:confluent_roots} extends the exact formula to repeated
roots.

For a \(2m\times2n\) rectangle,
\eqref{eq:dual_cauchy_toeplitz} and
Theorem~\ref{thm:domino_root_correction} relate the
Kasteleyn--Temperley--Fisher product, the product of tridiagonal Toeplitz
determinants in \eqref{eq:Z_cheb_prod}, the Toeplitz determinant derived from
the dual Cauchy identity, and a banded Toeplitz determinant.  The asymptotic
formulas
\eqref{eq:domino_root_limit}--\eqref{eq:domino_root_first} apply with \(m\)
fixed and \(n\to\infty\).  The classical product yields the explicit
fixed-width formulas \eqref{eq:Zmn_concrete} for boards of widths \(4\),
\(8\), and \(16\), while \eqref{eq:chessboard_count} recovers the number of
tilings of the \(8\times8\) chessboard.  Because the correction matrix grows with
\(m\) and \(q_*\to1\), the determinant comparison does not directly cover a
limit with fixed aspect ratio.  The unresolved quantity in that regime is the normalized
comparison \((-1)^{mn}Z_{m,n}/D_{2n}(F_m)\), rather than the asymptotics of
\(Z_{m,n}\) itself, which follow from \eqref{eq:Zmn}.

We close with two possible extensions.  First, the joint limit
\(m,n\to\infty\) requires an
analysis in which the bandwidth and symbol vary with the matrix size.  The
Riemann--Hilbert framework for determinants of sums of Toeplitz and Hankel
matrices
\cite{GharakhlooIts20} may provide a starting point, but it would have to
accommodate the growing boundary blocks.

Second, Toeplitz minors and skew Schur functions arise when boundary
conditions are modified or defects are prescribed in suitable models
invariant under translation \cite{BD02,GGT20}.  In matching models with
reflection symmetry, the factorization theorem \cite{Ciucu97} gives a related
factorization.  Petrov's extension of the
Borodin--Okounkov--Geronimo--Case identity to tilted Toeplitz minors
\cite{Petrov26}, with the column tilts denoted here by \(\zeta_j\), replaces
\(T_N(\varphi)\) by
\(\bigl[(\theta_i\zeta_j\varphi)_{i-j}\bigr]_{i,j=1}^N\), where \(\theta_i\)
and \(\zeta_j\) are analytic functions depending on the row and column,
respectively.  For \(F_m\), which is nonvanishing and has winding number zero,
the Borodin--Okounkov kernel
\(K\) has rank at most \(m\), as observed after
\eqref{eq:b_coeff_expansion}.
Petrov's formula assumes that the tilt functions are bounded and analytic in
their respective domains and that the Gram matrix \(\Gamma_{\zeta,\theta}\) is
invertible.

Under these assumptions, its Fredholm operator is
\(\Pi_VK|_V\), where \(\Pi_V\) is the oblique projection onto the subspace
\(V\).  Its rank is therefore at most \(m\).  The same
Sylvester identity used in the proof of Proposition~\ref{prop:finiteN_root}
expresses the Fredholm factor as the determinant of a matrix of dimension at
most \(m\).  This does not, however, give an explicit formula for general
choices of both families of tilts.  In the present approach, such a formula would require tilted
analogues
of the weights \eqref{eq:beta_weights} and the Cauchy matrix
\eqref{eq:C_matrix}, together with a closed form for
\(\det\Gamma_{\zeta,\theta}\).  Here \(\Gamma_{\zeta,\theta}\) is an
\(N\times N\) matrix, and \(\det\Gamma_{\zeta,\theta}\) is the separate Gram
factor.  The reduction above to a determinant of fixed dimension applies only to the
Fredholm factor.

Another question is whether the
factorization of Corollary~\ref{cor:parity_factorization} survives.  Although
\(T_N(F_m)\) is centrosymmetric, the tilted matrix
\(\bigl[(\theta_i\zeta_jF_m)_{i-j}\bigr]_{i,j=1}^N\) need not be
centrosymmetric.  For the same reversal argument to apply, one would need
conditions on \((\zeta,\theta)\)
ensuring that this matrix commutes with \(J_N\).  We do not know a convenient
condition of this kind.  These questions are distinct from the
finite-dimensional closure problem raised in \cite{Petrov26}, which concerns
symbols with rational \(z\,\partial_z\log\varphi(z)\) and polynomial or
rational tilts.

\paragraph{Acknowledgments.}
The author thanks Albrecht B\"ottcher for valuable comments and suggestions,
and Dario A.~Bini, Sergei Grudsky, Oleksiy Karlovych, Egor A.~Maximenko,
Beatrice Meini, and Leonid Petrov for useful correspondence.

This project began years ago and was initially developed without AI assistance.
ChatGPT was used for the numerical computations and a final check of the results.

\end{document}